\documentclass[useAMS]{chJS}

\usepackage[english]{babel}
\usepackage[utf8]{inputenx}
\usepackage{subfigure}
\usepackage{graphicx}
\usepackage{amssymb,amsmath,bm,bbm,booktabs}
\usepackage{units} 
\usepackage{enumerate}
\usepackage[normalem]{ulem}

\newtheorem{prop}{Proposition}

\DeclareMathOperator{\e}{\mathrm{e}}
\DeclareMathOperator{\NH}{\mathrm{NH}}
\DeclareMathOperator{\EGNH}{\mathrm{EGNH}}

\DeclareMathOperator{\E}{\mathbb{E}}

\begin{document}


  \markboth{VedoVatto et al.}{Chilean Journal of Statistics}

  \articletype{}

  \def \tr {\mathop{\rm tr}\nolimits}
  \def \etr {\mathop{\rm etr}\nolimits}
  \def \diag {\mathop{\rm diag}\nolimits}
  \def\build#1#2#3{\mathrel{\mathop{#1}\limits^{#2}_{#3}}}
  \newcommand{\Half}{\mbox{$\frac{1}{2}$}}
  \newcommand{\cuarto}{\mbox{$\frac{1}{4}$}}
  \def \R{\mathop{\rm Re}\nolimits}
  \newcommand {\findemo}{\hfill$\square$}

  \title{Some Computational and Theoretical Aspects of the Exponentiated Generalized Nadarajah-Haghighi Distribution}

  \author{VedoVatto, T.$^{\rm 1,2,\ast}$\thanks{$^\ast$Corresponding author. Email: thiagovedovatto@gmail.com
      \vspace{6pt}}, Nascimento, A.\ D.\ C.$^{\rm 2}$, Miranda Filho, W.\ R.$^{\rm 2}$, Lima, M.\ C.\ S.$^{\rm 2}$, Pinho, L.\ G.\ B.$^{\rm 3}$ and Cordeiro, G.\ M.$^{\rm 2}$\\\vspace{6pt}
    $^{\rm 1}${{Instituto Federal de Educa\c{c}\~ao, Ci\^encia e Tecnologia de Goi\'as, Goiânia, Brazil}}\\
    $^{\rm 2}${{Departamento de Estat\'istica, Universidade Federal de Pernambuco, Recife, Brazil}}
    $^{\rm 3}${{Departamento de Estat\'istica e Matemática Aplicada, Universidade Federal do Ceará, Fortaleza, Brazil}}
    \\\vspace{6pt}
    \received{Received: 00 Month 200x \sep  Accepted in final form: 00 Month 200x}
  }

  \maketitle

  \begin{abstract}
    Real data from applications in the survival context have required the use of more flexible models.
    A new four-parameter model called the Exponentiated Generalized Nadarajah-Haghighi (EGNH) distribution has been introduced in order to verify this requirement.
    We prove that its hazard rate function can be constant, decreasing, increasing, upside-down bathtub and bathtub-shape.
    Theoretical essays are provided about the EGNH shapes.
    It includes as special models the ex\-po\-nen\-tial, ex\-po\-nen\-tia\-ted ex\-po\-nen\-tial, Nadarajah-Haghighi's ex\-po\-nen\-tial and ex\-po\-nen\-tia\-ted Nadarajah-Haghighi distributions.
    We present a physical interpretation for the EGNH distribution and obtain some of its mathematical properties including shapes, moments, quantile, generating functions, mean deviations and Rényi entropy.
    We estimate its parameters by maximum likelihood, on which one of the estimates may be written in closed-form expression.
    This last result is assessed by means of a Monte Carlo simulation study.
    The usefulness of the introduced model is illustrated by means of two real data sets.
    We hope that the new distribution could be an alternative to other distributions available for modeling positive real data in many areas.
    \\
    \begin{classcode} 
      Primary 62E10	 \sep Secondary  62P30.
    \end{classcode}
  \end{abstract}

  \section{Introduction}

  In recent years, several ways of generating new continuous distributions have been proposed in survival analysis to provide flexibility and new forms for the \emph{hazard rate function} (hrf).
  A detailed study about `the evolution of the methods for generalizing classic distributions' was made by \cite{LeeFamoyeAlzaatreh-Methodsgeneratingfamilies-2013}.
  Most of them are special cases of the \emph{T-X} class defined by \citet{AlzaatrehLeeFamoye-newmethodgenerating-2013}.
  This class of distributions extends some recent families such as
  the beta-G pioneered by \citet{EugeneLeeFamoye-BetanormalDistribution-2002},
  the gamma-G defined by \citet{ZografosBalakrishnan-familiesbetaand-2009},
  the Kw-G fa\-mi\-ly proposed by \citet{CordeiroCastro-newfamilygeneralized-2011}
  and the Weibull-G introduced by \citet{BourguignonSilvaCordeiro-WeibullGFamily-2014}.

  The process to construct a generator based on the \emph{T-X} class is given as follows. Let $r(t)$ be the probability density function (pdf) of a random variable $T \in [a,b]$ for $-\infty<a<b<\infty$ and let $W[G(x)]$ be a mapping under the following conditions:
  \begin{itemize}
    \item $W[G(x)] \in [a,b]$;
    \item $W[G(x)]$ is monotonically non-decreasing and differentiable;
    \item $W[G(x)] \rightarrow a$ as $x \rightarrow -\infty$ and $W[G(x)] \rightarrow b$ as $x \rightarrow \infty$.
  \end{itemize}

  The new \textit{T-X} generator has \emph{cumulative distribution function} (cdf) and pdf given, respectively, by
  \begin{equation*} 
    F(x) = \int_{a}^{W[G(x)]} r(t) dt\quad\text{and}\quad  f(x) = \left\{ \frac{d}{dx}W[G(x)] \right\}r\left\{ W[G(x)] \right\}.
  \end{equation*}

  The choice of $W(\cdot)$ is intrinsically associated with the cdf of $T$.
  For example, if $T$ follows a beta distribution, $W(\cdot)$ can be the identity function, leading to the beta-G fa\-mi\-ly.
  On the other hand, if $T$ follows a gamma distribution, $W[G(x)]$ must be a function such that $W{:}[0,1]\rightarrow[0,\infty)$.
  Examples of functions that satisfy this restriction are: $W[G(x)] = -\log[1-G(x)]$, $W[G(x)] = -\log[G(x)]$ and $W[G(x)] = G(x)/[1-G(x)]$.

  \citet{CordeiroOrtegaCunha-ExponentiatedGeneralizedClass-2013} proposed an interesting fa\-mi\-ly in the \emph{T-X} class inspired by the second type Kumaraswamy distribution. 
  We refer to this fa\-mi\-ly as the \emph{ex\-po\-nen\-tia\-ted ge\-ne\-ra\-li\-zed} (EG) fa\-mi\-ly, which has cdf and pdf (for $x>0$) given, respectively, by
  \begin{equation}\label{eq:cdfgerador}
    F(x) = \{1 - [1-G(x)]^\alpha\}^\beta
  \end{equation}
  and
  \begin{equation}\label{eq:pdfgerador}
    f(x) = \frac{\alpha\beta g(x)[1-G(x)]^{\alpha -1}}{\{1 -[1 - G(x)]^\alpha\}^{1-\beta}},
  \end{equation}
  where $ \alpha >0$ and $ \beta>0 $ are two additional shape parameters that can control both tail weights.
  One can note that equations \eqref{eq:cdfgerador} and \eqref{eq:pdfgerador} do not involve any complicated function, which is an advantage when compared to the beta and gamma families. 
  Further, the EG fa\-mi\-ly is nothing more than a sequential application of the Lehmann type 2 alternative to the ex\-po\-nen\-tia\-ted class. 
    Setting $\alpha = 1$ gives the ex\-po\-nen\-tia\-ted-G (Lehmann type 1) class defined by \citet{GuptaGuptaGupta-Modelingfailuretime-1998}.

  The best benefit of the EG fa\-mi\-ly is its ability to fitting skewed data.
  This class is also dual of the Kw-G fa\-mi\-ly and has similar properties of the beta-G fa\-mi\-ly and some advantages in terms of tractability. 
  Some special models in this class have been studied recently.
  The EG-inverse Gaussian  by \citet{LemonteCordeiro-exponentiatedgeneralizedinverse-2011},
  the EG-generalized gamma by \citet{SilvaGomes-SilvaRamosCordeiro-NewExtendedGamma-2015},
  the EG-inverse Weibull   by \citet{ElbatalMuhammed-ExponentiatedGeneralizedInverse-2014}
  and the EG-Gumbel        by \citet{AndradeRodriguesBourguignonCordeiro-exponentiatedgeneralizedGumbel-2015}
  distributions are special cases obtained by taking $G(x)$ to be the cdf of the inverse Gaussian, generalized gamma, inverse Weibull and Gumbel distributions, respectively.
  Hence, each new EG distribution can be generated from a specified $G$ distribution.

  A new generalization of the ex\-po\-nen\-tial law as an alternative to the gamma and Weibull models has been proposed by \citet{NadarajahHaghighi-extensionexponentialdistribution-2011}, which model has cdf given by (for $x>0$)
  \begin{equation}\label{eq:cdfbaseline}
    G(x) = 1-\e^{1-(1+a x)^b},
  \end{equation}
  where $a>0$ and $b>0$. Its pdf reduces to
  \begin{equation}\label{eq:pdfbaseline}
    g(x) = a b (1+a x)^{b-1}\e^{1-(1+a x)^b}.
  \end{equation}

  This model is called the Nadarajah-Haghighi's ex\-po\-nen\-tial (denoted by $Z\sim\NH(a,b)$) because the ex\-po\-nen\-tial is a special case when $b=1$.
  This bi-parametric model has been used for modeling lifetime data.
  The pdf \eqref{eq:pdfbaseline} is always monotonically decreasing with $g(0)=ab$.
  \citet{NadarajahHaghighi-extensionexponentialdistribution-2011} pointed out that this distribution has an attractive feature of always having the zero mode. 
  They also showed that larger values of $b$ lead to faster decay of the upper tail.  
  Additionally, the hrf can be monotonically increasing for $b>1$, and monotonically decreasing for $b<1$.
  For $b=1$, the hrf becomes constant. 
  So, the major weakness of the NH distribution is its inability to accommodate non-monotone hrfs (i.e. bathtub-shaped and upside-down bathtub).


  Several extensions of the NH model have been proposed in recent years such as the following distributions:
  the generalized power Weibull (GPW), Kumaraswamy generalized power Weibull (KGPW),
  ex\-po\-nen\-tia\-ted (Lehmann type 1) Nadarajah-Haghighi (ENH),
  gamma (Zografos-Balakrishnan type) Nadarajah-Haghighi (GNH),
  Poisson gamma (also of Zografos-Balakrishnan type) Nadarajah-Haghighi (PGNH),
  transmuted (quadratic random transmuted map) Nadarajah-Haghighi (TNH),
  Kumaraswamy Nadarajah-Haghighi (KNH),
  ex\-po\-nen\-tia\-ted (Lehmann type 2) Nadarajah-Haghighi (E2NH) (submodel of the KNH),
  modified Nadarajah-Haghighi (MNH),
  Marshall-Olkin Nadarajah-Haghighi (MONH)
  and beta Nadarajah-Haghighi (BNH). 
  These distributions and their corres\-pon\-ding authors are listed in
  Table \ref{table:RecentNadarajah-Haghighiextensions}.

  \begin{table}
    \caption{Nadarajah-Haghighi extensions}
    \centering
    \begin{tabular}{l|l}
    	\hline
    	Distribution & Author(s)                                                                                                                                  \\ \hline
    	GPW          & \citet{NikulinHaghighi-ChiSquaredTest-2006,NikulinHaghighi-powergeneralizedWeibull-2009}                                                   \\
    	ENH          & \citet{Lemonte-newexponentialtype-2013} and \citet{Abdul-Moniem-ExponentiatedNadarajahand-2015}                                            \\
    	GNH          & \citet{OrtegaLemonteSilvaCordeiro-Newflexiblemodels-2015} and \citet{BourguignonCarmoLeaoNascimentoPinhoCordeiro-newgeneralizedgamma-2015} \\
    	PGNH         & \citet{OrtegaLemonteSilvaCordeiro-Newflexiblemodels-2015}                                                                                  \\
    	TNH          & \citet{AhmedMuhammedElbatal-NewClassExtension-2015}                                                                                        \\
    	KNH          & \citet{Lima-TheHalf-normalGeneralizedFamilyAndKumaraswamyNadarajah-haghighiDistribution-2015}                                              \\
    	MNH          & \citet{El-DamceseRamadan-StudiesPropertiesand-2015}                                                                                        \\
    	MONH         & \citet{LemonteCordeiroMoreno-Arenas-newusefulthree-2016}                                                                                   \\
    	KGPW         & \citet{SelimBadr-KumaraswamyGeneralizedPower-2016}                                                                                         \\
    	BNH          & \citet{Dias-NewContinuousDistributionsAppliedToLifetimeDataAndSurvivalAnalysis-UniversidadeFederaldePernambuco-2016}                       \\ \hline
    \end{tabular}
    \label{table:RecentNadarajah-Haghighiextensions}
  \end{table}

  In this paper, we obtain a natural extension of the NH distribution, which we refer to as the \emph{exponentiated generalized Nadarajah-Haghighi} (EGNH) distribution. 
  The first motivation for introducing the new distribution is based on the fact that it has simple expressions for the pdf and hrf and, as a consequence, a simpler statistical inference process.
  Second, the flexibility of the new distribution to describe complex positive real data is concluded since this hrf can present constant, decreasing, increasing, unimodal, bathtub-shaped and upside-down bathtub forms.
  Due to the great flexibility of its hrf, it can provide a good alternative to many existing lifetime distributions to model real data. 
  Therefore, the beauty and importance of the new distribution lies in its ability to model monotone as well as nonmonotone hrfs, which are quite common in reliability and biological studies.

  So, our main aim is to propose a new four-parameter distribution, which extends the ex\-po\-nen\-tial, NH and ENH distributions, with the hope that it may provide a `better fit' compared to other lifetime distributions in certain practical situations.
  Additionally, we provide a comprehensive account of the mathematical properties of the introduced distribution.
  The formulae related with the new distribution are simple and manageable, and, with the use of modern computer resources and their numerical capabilities, the proposed distribution may prove to be an useful addition to the arsenal of applied statisticians in areas like biology, medicine, economics, reliability, engineering, among others.

  The paper is outlined as follows.
  In Section \ref{sec:Theproposedmodel}, we introduce the new distribution, some of its basic properties and provide plots of the pdf and hrf.
  Section \ref{sec:QuantileMeasures} deals with nonstandard measures of skewness and kurtosis.
  Linear representations for the pdf and cdf are presented in Section \ref{sec:DensityExpansions} and explicit expressions for the moments are provided in Section \ref{sec:Moments}.
  Section \ref{sec:GeneratingFunction} is devoted to the \emph{moment generating function} (mgf).
  The mean deviations are obtained in Section \ref{sec:MeanDeviations}.
  The Rényi entropy is derived in Section \ref{sec:RenyiEntropy}.
  In Section \ref{sec:OrderStatistics}, we present the order statistics.
  A procedure to calculate the \emph{maximum likelihood estimators} (MLEs) of the model parameters is investigated in Section \ref{sec:MLE}.
  A simulation study illustrating the convergence and consistency of the MLEs is addressed in Section \ref{sec:SimulationStudies}.
  Two applications to real data are provided in Section \ref{sec:Applications}.
  Finally, Section \ref{sec:Concludingremarks} offers some concluding remarks.


  \section{The proposed model}\label{sec:Theproposedmodel}

  The EGNH cdf is obtained by applying \eqref{eq:cdfbaseline} in Equation \eqref{eq:cdfgerador}
  \begin{equation}\label{eq:cdfmodelo}
    F(x) = \left\{1-\left[\e^{1-(1+a x)^b}\right]^{\alpha }\right\}^{\beta },\quad x>0,
  \end{equation}
  where $\alpha,\beta,a,b>0$ and $b\ne 1$.
  The EGNH pdf can be obtained by applying \eqref{eq:cdfbaseline} and \eqref{eq:pdfbaseline} in \eqref{eq:pdfgerador} or by differentiating \eqref{eq:cdfmodelo}:
  \begin{equation}\label{eq:pdfmodelo}
    f(x)= \frac{a \alpha  b \beta  (1+a x)^{b-1} \left[\e^{1-(1+a x)^b}\right]^{\alpha }}{\left\{1-\left[\e^{1-(1+a x)^b}\right]^{\alpha }\right\}^{1-\beta}}.
  \end{equation}

  Henceforth, we denote by $X\sim\EGNH(\alpha,\beta,a,b)$ a random variable having pdf \eqref{eq:pdfmodelo}.
  So, the EGNH distribution is obtained by adding two shape parameters $\alpha$ and $\beta$ to the NH distribution.
  It is necessary to consider $b\ne 1$, an additional restriction to the parametric space, to give identifiability to the generated model.

  The EGNH distribution has also an attractive physical interpretation whenever $\alpha$ and $\beta$ are positive integers.
  Consider a device based on $\beta$ independent components in a parallel system.
  Further, each component is defined in terms of $\alpha$ independent subcomponents identically distributed following the NH model in a series system.
  Assume that the device fails if all $\beta$ components fail and each component fails if any subcomponent fails.
  Let $X_{j1},\ldots,X_{j\alpha}$ denote the lifetimes of the subcomponents within the $j$th component, $j=1,\ldots,\beta$, with common cdf given by \eqref{eq:cdfbaseline}.
  Let $X_j$ denote the lifetime of the $j$th component and let $X$ denote the lifetime of the device.
  Thus, the cdf of $X$ is given by
  \begin{align*}\label{eq:PhysicalInterpretation}
    P(X\le x)=& P(X_1\le x,\ldots,X_\beta\le x)=P(X_1\le x)^\beta=\left[1-P(X_1>x)\right]^\beta\\
    =& \left[1-P(X_{11}>x,\ldots,X_{1\alpha}>x)\right]^\beta=\left[1-P(X_{11}>x)^\alpha\right]^\beta\\
    =& \left\{1-\left[1-P(X_{11}\le x)\right]^\alpha\right\}^\beta.
  \end{align*}
  So, the lifetime of the device obeys the EGNH distribution.

  The quantile function (qf) of $X$ is determined by inverting \eqref{eq:cdfmodelo} (for $0<p<1$) as
  \begin{equation}\label{eq:quantile function}
    Q(p)=\frac{1}{a}\left\{\left\{1-\log\left[\left(1-p^{\frac{1}{\beta}}\right)^{\frac{1}{\alpha}}\right]\right\}^{\frac{1}{b}}-1\right\}.
  \end{equation}
  Clearly, the median of the new distribution follows by setting $p=1/2$ in \eqref{eq:quantile function}. We can also use \eqref{eq:quantile function} for simulating EGNH random variables by the \emph{inverse transform sampling method}, which works as follows:
  \begin{enumerate}
    \item Let $u$ be an outcome of uniform distribution in the interval $[0,1]$.
    \item $x=Q(u)$ is them an outcome of the EGNH distribution.
  \end{enumerate}




  Finally, the \emph{survival function} (sf) and hrf of $X$ are given, respectively, by
  \begin{align}
    S(x) & = 1-\left\{1-\left[\e^{1-(1+a x)^b}\right]^{\alpha }\right\}^{\beta }\,\,\,\text{and}\nonumber\\
    h(x) & = \frac{a \alpha  b \beta  (1+a x)^{b-1} \left[\e^{1-(1+a x)^b}\right]^{\alpha } }{1-\left\{1-\left[\e^{1-(1+a x)^b}\right]^{\alpha }\right\}^{\beta }\left\{1-\left[\e^{1-(1+a x)^b}\right]^{\alpha }\right\}^{1-\beta}}\label{eq:hazard rate}.
  \end{align}

  The reverse hrf of $X$ is given by
  \begin{equation*}\label{eq:reverhaf}
    r(x)= \frac{a\alpha b\beta(1+ax)^{b-1}\left[\e^{1-(1+ax)^b}\right]^\alpha}{\left\{1-\left[\e^{1-(1+ax)^b}\right]^\alpha\right\}^\beta}= \frac{\beta}{\left[F_{E2NH}(x)\right]^\beta}\,f_{E2NH}(x)
  \end{equation*}
  where $F_{E2NH}(x)$ and $f_{E2NH}(x)$ are, respectively, the cdf and pdf of E2NH (Lehmann type 2) distribution, which is a special case of the model proposed by \citet{Lima-TheHalf-normalGeneralizedFamilyAndKumaraswamyNadarajah-haghighiDistribution-2015} and \cite{SelimBadr-KumaraswamyGeneralizedPower-2016}. In the following, we address limit distributions, shape and special cases of the EGNH model.

  \subsection{Limiting behavior of the density}
  The pdf \eqref{eq:pdfmodelo} can take various forms depending on the values of the $\alpha$ and $\beta$ shape parameters.
  It is easy to verify that
  \begin{align*}\label{eq:LimitingBehaviorToZero}
    \lim_{x\to 0}f(x)=
    \begin{cases}
      \infty,   & \beta<1 \\
      ab\alpha, & \beta=1 \\
      0,& \beta>1
    \end{cases} \quad\text{and}\quad \lim_{x\to\infty}f(x)=0.
  \end{align*}

  \subsection{Shapes}

  In Figures \ref{fig:density} and \ref{fig:hazard}, we plot some possible shapes of the pdf \eqref{eq:pdfmodelo} and hrf \eqref{eq:hazard rate}, respectively, for selected parameter values.
  The pdf and hrf of $X$ can take various forms depending on the parameter values.
  It is evident that the EGNH distribution is much more flexible than the $\NH$ distribution. 
  So, the new model can be very useful in many practical situations for modeling positive real data.

  \begin{figure}[htb!]
    \centering
    \includegraphics[width=\linewidth]{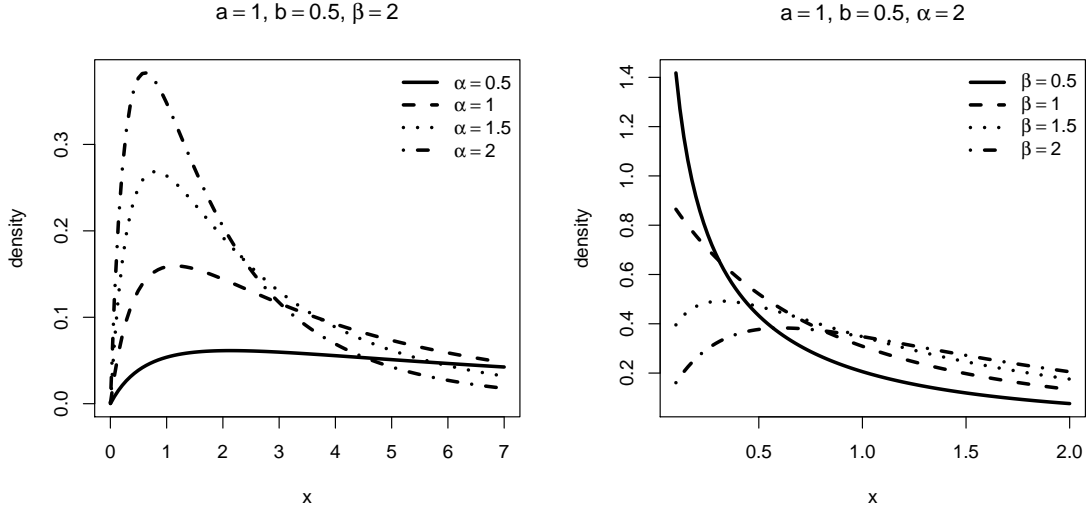}
    \caption{Plots of the pdf \eqref{eq:pdfmodelo} for some parameter values.}
    \label{fig:density}
  \end{figure}

  \begin{figure}[htb!]
    \centering
    \includegraphics[width=\linewidth]{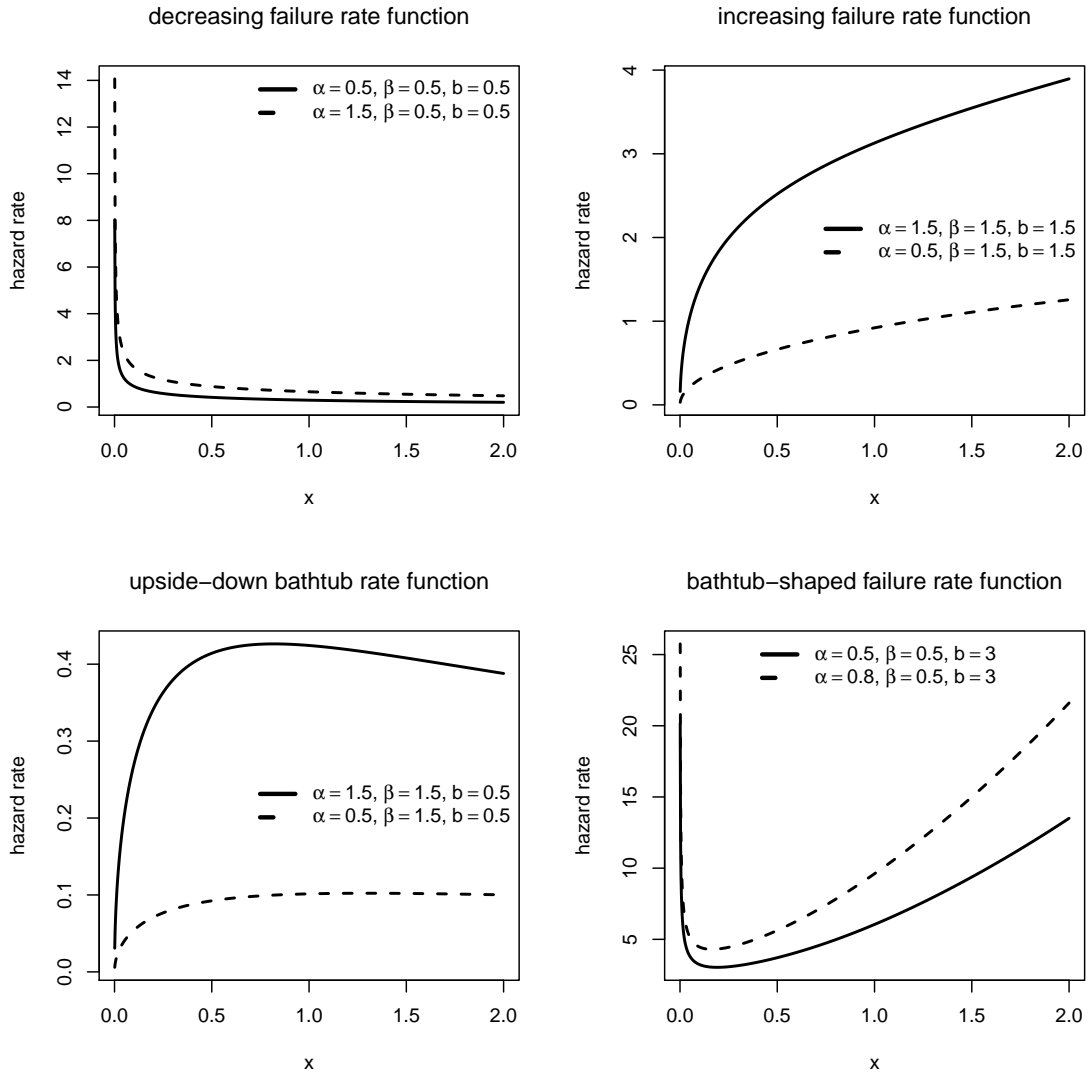}
    \caption{Plots of the hrf \eqref{eq:hazard rate} for some parameter values and $a=1$}
    \label{fig:hazard}
  \end{figure}

  It is easy to prove that
  \begin{equation*}\label{eq:LogDensity1d}
    \frac{d \log\left[f(x)\right]}{dx} = \frac{a \alpha  b ( \beta -1 )  (1+a x)^{b-1} \left(\e^{1-(1+a x)^b}\right)^{\alpha }}{1-\left(\e^{1-(1+a x)^b}\right)^{\alpha }}-a \alpha  b (1+a x)^{b-1}+\frac{a (b-1)}{1+a x}.
  \end{equation*}

  The mode of the pdf \eqref{eq:pdfmodelo} is the solution of the equation
  \begin{equation*}\label{eq:ModesEquation}
    \frac{d\log[f(x)]}{dx}=0.
  \end{equation*}

  \begin{prop}\label{prop:pdflogconvex}
    For any $\alpha>0$ and $a>0$, the EGNH pdf is log-convex if $\beta<1$ and $b<1$, and it is log-concave if $\beta>1$ and $b>1$.
  \end{prop}
  \begin{proof}
    The proof is analogous to that one given by \citet{Lemonte-newexponentialtype-2013}.
    Let $z=(1+ax)^b$.
    Then, $x>0$ implies $z>1$.
    Inverting for $x$, we obtain $x=(z^{1/b}-1)/a$.
    The EGNH pdf rewritten as a function of $z$ is given by
    \begin{equation}\label{eq:newpdfmodelo}
      \phi(z)=\frac{a \alpha  b \beta  z^{\frac{b-1}{b}} \e^{\alpha } \left(1-\e^{\alpha -\alpha  z}\right)^{\beta }}{\e^{\alpha  z}-\e^{\alpha }},\quad z>1.
    \end{equation}

    The result follows by nothing that the second derivative of $\log[\phi(z)]$ is
    \begin{equation*}\label{eq:newpdfmodelosecondderivative}
      \frac{d^2 \log[\phi(z)]}{dz^2}=-\left\{\frac{b-1}{bz^2}+\frac{1}{4} \alpha ^2 (\beta -1) \text{csch}^2\left[\frac{1}{2} (\alpha -\alpha  z)\right]\right\}.
    \end{equation*}
  \end{proof}

  \begin{prop}\label{prop:hrfshapes}
    For any $\alpha>0$ and $a>0$, the EGNH model has an increasing hrf if $\beta>1$ and $b>1$ and it has a decreasing hrf if $\beta<1$ and $b<1$. 
    It is constant if $\beta=b=1$.
  \end{prop}
  \begin{proof}
    First, note that $\beta=b=1$ implies $h(x)=a\alpha$ (constant hrf).

    Next, let $z=(1+ax)^b$. 
    Notice that $z>1$ (for $x>0$) and $x=\nicefrac{(z^{1/b}-1)}{a}$.
    Rewriting the EGNH hrf as a function of $z$, we have 
    \begin{equation}\label{eq:newhrfmodelo}
      \eta(z)= \frac{a \alpha  b \beta z^{\frac{b-1}{b}} \left[1-\e^{\alpha(1-z)}\right]^{\beta }}{\left[1-\e^{\alpha(z-1)}\right] \left\{\left[1-\e^{\alpha(1-z)}\right]^{\beta }-1\right\}},\quad z>1.
    \end{equation}

    Next, taking the first derivative of the logarithm of both sides of equation \eqref{eq:newhrfmodelo}, we obtain
    \begin{equation*}\label{newhrfmodelofirstderivative}
      \frac{d \log[\eta(z)]}{dz}= \frac{\eta'(z)}{\eta(z)}=\frac{b-1}{b z}+\alpha  \left\{\frac{\beta +\left[1-\e^{\alpha(1-z)}\right]^{\beta }-1}{\left[1-\e^{\alpha(z-1)}\right] \left\{\left[1-\e^{\alpha(1-z)}\right]^{\beta }-1\right\}}-1\right\},
    \end{equation*}
    which implies that
    \begin{equation*}
      \eta'(z)=\eta(z)\left\{ \frac{b-1}{bz}+\alpha\left\{\frac{\beta +\left[1-\e^{\alpha(1-z)}\right]^{\beta }-1}{\left[1-\e^{\alpha(z-1)}\right] \left\{\left[1-\e^{\alpha(1-z)}\right]^{\beta }-1\right\}}-1\right\}\right\}.
    \end{equation*}
    Thus, $\eta'(z)$ has the same sign of
    \begin{equation*}\label{eq:auxiliarfunction}
      \xi(z)=\frac{b-1}{bz}+\alpha\left\{\frac{\beta +\left[1-\e^{\alpha(1-z)}\right]^{\beta }-1}{\left[1-\e^{\alpha(z-1)}\right] \left\{\left[1-\e^{\alpha(1-z)}\right]^{\beta }-1\right\}}-1\right\},\quad z>1,
    \end{equation*}
    since $\eta(z)>0$ for $z>1$. 
    Note that $\xi(z)<0$ if $b<1$ and $\beta<1$ and $\xi(z)>0$ if $b>1$ and $\beta>1$.

  \end{proof}
  
  \noindent The last proposition is important because it anticipates behaviors of device lifetime on practical situations.

  \subsection{Special distributions}

  In Table \ref{table:Special Distributions}, one can observe some EGNH special models.
  As expected, the ex\-po\-nen\-tial distribution is obtained as a special case, but surprisingly by two distinct ways. 
  The ex\-po\-nen\-tia\-ted ex\-po\-nen\-tial (EE) is also obtained by two different settings.
  Any baseline distribution will also be a sub-model when all the additional parameters of the EG fa\-mi\-ly are equal to one, and then the NH distribution is listed. 
  Finally, the ENH (Lehmann type 1) and E2NH (Lehmann type 2) models are examples of distributions discussed early, which are really special cases of the EGNH model.

  \begin{table}
    \caption{Some special distributions}
    \centering
    \begin{tabular}{l|c|c|c|c}
      \hline
      Model        & $ \alpha $ & $ \beta $ & $ a $ & $ b $ \\ \hline
      ENH          &     1      &     -     &   -   &   -   \\
      E2NH         &     -      &     1     &   -   &   -   \\
      $\NH$        &     1      &     1     &   -   &   -   \\
      EE           &     1      &     -     &   -   &   1   \\
      EE           &     -      &     -     &   1   &   1   \\
      exponential  &     1      &     1     &   -   &   1   \\
      ex\-po\-nen\-tial  &     -      &     1     &   1   &   1   \\ \hline
    \end{tabular}
    \label{table:Special Distributions}
  \end{table}

 \section{Quantile measures} \label{sec:QuantileMeasures}
 
 The effects of the parameters $\alpha$ and $\beta$ on the skewness and kurtosis of $X$ can be considered based on the qf given by \eqref{eq:quantile function}. 
 The Bowley skewness (B), proposed by \citet{KenneyKeeping-Mathematicsofstatistics-1962}, and the Moors kurtosis (M), proposed by \citet{Moors-QuantileAlternativeKurtosis-1988}, are defined, respectively, by
 \begin{equation*}\label{eq:BowleySkewness}
 B= \frac{Q( \nicefrac{3}{4})-2Q( \nicefrac{1}{2})+Q( \nicefrac{1}{4})}{Q(\nicefrac{3}{4})-Q(\nicefrac{1}{4})}
 \quad\text{and}\quad
 M= \frac{Q( \nicefrac{7}{8})-Q( \nicefrac{5}{8})+ Q( \nicefrac{3}{8}) - Q( \nicefrac{1}{8})}{Q( \nicefrac{3}{4})-Q(\nicefrac{1}{4})}.
 \end{equation*}
 The $B$ and $M$ measures are less sensitive to outliers and they exist even for distributions without moments.
 
 In Figures \ref{fig:BowleySkewness} and \ref{fig:MoorsKurtosis}, we plot the measures $B$ and $M$ for some parameter values of the EGNH distribution as functions of $\alpha$ and $\beta$ (for $a$ and $b$ fixed).
 These plots indicate that both measures can be very sensitive on these shape parameters and reveal the importance of the proposed distribution.
 
 \begin{figure*}[htb!] 
 	\centering
 	\subfigure[$a=0.5,b=2.0$]{
 		\includegraphics[width=0.45\textwidth]{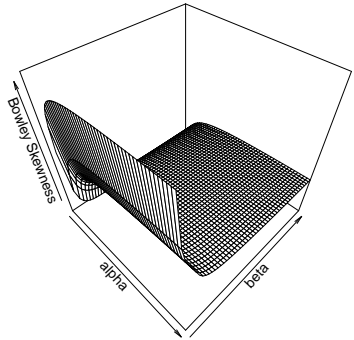}
 		\label{fig:BowleySkewness1}
 	}
 	\subfigure[$a=0.5,b=0.5$]{
 		\includegraphics[width=0.45\textwidth]{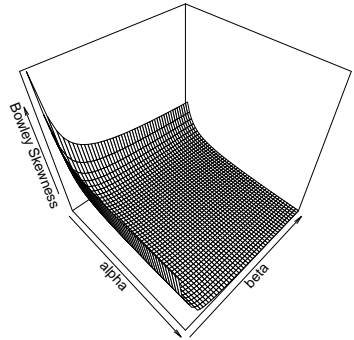}
 		\label{fig:BowleySkewness2}
 	}
 	\caption{The Bowley skewness of the EGNH distribution as function of $\alpha$ and $\beta$.}
 	\label{fig:BowleySkewness}
 \end{figure*}
 
 \begin{figure*}[htb!] 
 	\centering
 	\subfigure[$a=0.5,b=2.0$]{
 		\includegraphics[width=0.45\textwidth]{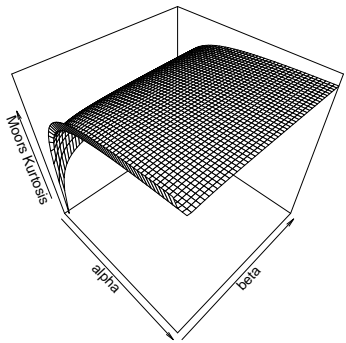}
 		\label{fig:MoorsKurtosis1}
 	}
 	\subfigure[$a=0.5,b=0.5$]{
 		\includegraphics[width=0.45\textwidth]{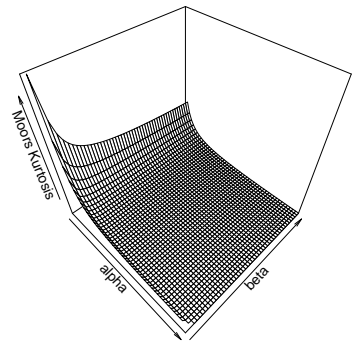}
 		\label{fig:MoorsKurtosis2}
 	}
 	\caption{The Moors kurtosis of the EGNH distribution as function of $\alpha$ and $\beta$.}
 	\label{fig:MoorsKurtosis}
 \end{figure*}
 
 
  \section{Linear Representations} \label{sec:DensityExpansions}
  Some useful linear representations for \eqref{eq:cdfmodelo} and \eqref{eq:pdfmodelo} can be derived using the concept of ex\-po\-nen\-tia\-ted distributions.
	The properties of these distributions have been studied by many authors in recent years, see \citet{MudholkarSrivastava-ExponentiatedWeibullfamily-1993} and \citet{MudholkarSrivastavaFreimer-ExponentiatedWeibullFamily-1995} for exponentiated Weibull, \citet{GuptaGuptaGupta-Modelingfailuretime-1998} and \citet{GuptaKundu-ExponentiatedExponentialFamily-2001} for exponentiated exponential, \citet{NadarajahGupta-productParetodistribution-2007} for exponentiated gamma and \citet{CordeiroOrtegaSilva-exponentiatedgeneralizedgamma-2011} for exponentiated generalized gamma distributions.
  For an arbitrary baseline cdf $G(x)$, a random variable is said to have the ex\-po\-nen\-tia\-ted-G (``exp-G'' for short) distribution with power parameter $a>0$, say $Y\sim\operatorname{exp-G}(a)$, if its cdf and pdf are $H_a(x)=G(x)^a$ and $h_a(x)=ag(x)G(x)^{a-1}$, respectively.
  For any real non-integer $\beta$, we also consider the power series
  \begin{equation}\label{eq:PowerSeriesExpansion}
    (1-z)^{\beta}=\sum_{i=0}^{\infty}(-1)^i\binom{\beta}{i}z^i,
  \end{equation}
  which holds for $|z|<1$.

  Using expansion \eqref{eq:PowerSeriesExpansion} in equation \eqref{eq:cdfmodelo}, we can express the EGNH cdf as
  \begin{equation*}\label{eq:cdfExpansion-step1}
    F(x)=\sum_{i=0}^{\infty} (-1)^i\,\binom{\beta}{i}\left\{1-\left[1-\e^{1-(1+ax)^b}\right]\right\}^{i\alpha}
  \end{equation*}
  and applying again \eqref{eq:PowerSeriesExpansion} in the equation above gives
  \begin{equation*}\label{eq:cdfExpansion-step2}
    F(x)=\sum_{i,j=0}^{\infty} (-1)^{i+j} \binom{\beta}{i} \binom{i\alpha}{j}\left[1-\e^{1-(1+ax)^b}\right]^j.
  \end{equation*}

  Then, the following proposition holds. 
  
  \begin{prop}
	  The EGNH cdf can be rewritten as
	  \begin{equation}\label{eq:cdfExpansion}
	    F(x)=\sum_{j=0}^{\infty}w_j\,H_j(x),
	  \end{equation}
	  where the coefficients (for $j \ge 0$) are given by
	  \begin{equation*}\label{eq:cdfCoefficientsExpansion}
	    w_j=w_j(\alpha,\beta)=\sum_{i=0}^{\infty}(-1)^{i+j}\binom{\beta}{i}\binom{i\alpha}{j}
	  \end{equation*}
	  and $H_j(x)=\{1-\exp[1-(1+ax)^b]\}^j$ is the ENH$(j,a,b)$ cdf (for $j \ge 0$).
  \end{prop}
  	  We can prove using Mathematica that $\sum_{j=0}^{\infty}w_j=1$ as expected.

  By differentiating \eqref{eq:cdfExpansion}, we obtain
  \begin{equation}
    f(x)=\alpha\beta g(x)\sum_{j=0}^\infty t_j\,G(x)^j
    \label{eq:pdfExpansion}
  \end{equation}
  where the coefficients (for $j \ge 0$) are given by
  \begin{equation}
    t_j=t_j(\alpha,\beta)=\sum_{i=0}^{\infty}(-1)^{i+j}\binom{\beta-1}{i}\binom{\alpha(i+1)-1}{j}
    \label{eq:pdfCoefficientsExpansion}
  \end{equation}
  Further, the corollary bellow follows:
  
  \begin{prop}
  Equation \eqref{eq:pdfExpansion} can be rewritten as
  \begin{equation}
    f(x)=\sum_{j=0}^{\infty}u_j\,h_{j+1}(x),
    \label{eq:pdfExpansionAlternative}
  \end{equation}
  where the coefficients are
  \begin{equation*}
    u_j=u_j(\alpha,\beta)=\frac{\alpha\beta}{j+1}\,t_j
    \label{eq:pdfCoefficientsExpansionAlternative}
  \end{equation*}
  and
  \begin{equation}
    h_{j+1}(x)= \frac{ab(j+1)(1+ax)^b\exp[1-(1+ax)^b]}{\{1-\exp[1-(1+ax)^b]\}^{-j}}
    \label{eq:pdfExponentiatedNH}
  \end{equation}
  is the ENH$(j+1,a,b)$ pdf.
  \end{prop}
  
  Equation \eqref{eq:pdfExpansionAlternative} reveals that the EGNH pdf is an infinite linear combination of ENH pdfs.
  Thus, some structural properties of the EGNH distribution, such as the ordinary and incomplete moments and generating function, can be determined from those of the ENH model, which have been studied by \citet{Lemonte-newexponentialtype-2013} and \citet{Abdul-Moniem-ExponentiatedNadarajahand-2015}.

  \section{Moments}\label{sec:Moments}
  In this section, we derive the moments of the EGNH distribution as well as other related measures.
  Let $Y\sim\NH(a,b)$ and $X\sim\EGNH(\alpha,\beta,a,b)$.
  \citet{CordeiroOrtegaCunha-ExponentiatedGeneralizedClass-2013} proved that the moments of the EG distribution can be expressed as an infinite weighted sum of the $(r,j)$th probability weighted moments $\tau_{r,j}$ of the baseline random variable.
  Following this result, we can write from \eqref{eq:pdfExpansion}
  \begin{equation*}
    \mu_r'=\E(X^r)=\alpha\beta\sum_{j=0}^\infty t_j\,\tau_{r,j},
    \label{eq:momentsEGNH}
  \end{equation*}
  where $t_j$ is given by \eqref{eq:pdfCoefficientsExpansion} and
  \begin{equation*}
    \tau_{r,j}=\E[Y^r\,G(Y)^j]=\int_{-\infty}^{\infty}x^r\,G(x)^j\,g(x)dx.
    \label{eq:PWMNadarajahHaghighi}
  \end{equation*}

  By simple algebraic manipulations,
  \begin{equation}
    \tau_{r,j}=ab\int_0^\infty \frac{x^r(1+ax)^{b-1}\exp\left[1-(1+ax)^b\right]}{\left\{ 1-\exp\left[1-(1+ax)^b\right] \right\}^{-j}}\,dx.
    \label{eq:DedutionStep1}
  \end{equation}

  It is clear that $\left|\exp\left[1-\left(1+ax \right)^b\right]\right|<1$, and then applying \eqref{eq:PowerSeriesExpansion}
  in \eqref{eq:DedutionStep1} gives
  \begin{equation*}
    \tau_{r,j}=ab\,\sum_{k=0}^\infty(-1)^k\binom{j}{k}\int_0^\infty x^r (1+ax)^{b-1}\left\{\exp\left[1-(1+ax)^b\right] \right\}^{k+1}\,dx.
    \label{eq:DedutionStep2}
  \end{equation*}
  Then, we can write
  \begin{equation*}
    \tau_{r,j}=ab\,\sum_{k=0}^\infty(-1)^k\binom{j}{k}\e^{k+1}\int_0^\infty x^r(1+ax)^{b-1}\exp\left[-(k+1)(1+ax)^b\right]\,dx.
    \label{eq:DedutionStep3}
  \end{equation*}
  By setting $y=(k+1)(1+ax)^b$, $\tau_{r,j}$ reduces to
  \begin{equation*}
    \tau_{r,j}=\frac{1}{a^r}\sum_{k=0}^\infty(-1)^{ k+r }\binom{j}{k}\frac{\e^{k+1}}{k+1}\int_{k+1}^\infty\left[ 1-\left( \frac{y}{k+1}\right)^{\nicefrac{1}{b}} \right]^r\e^{-y}\,dy.
    \label{eq:DedutionStep4}
  \end{equation*}

  Taking $r$ integer and applying the binomial expansion, the last expression becomes
  \begin{equation*}
    \tau_{r,j}=\frac{1}{a^r}\sum_{k=0}^\infty\sum_{l=0}^r(-1)^{k+r+l}\binom{j}{k}\binom{r}{l}\frac{\e^{k+1}}{( k+1 )^{\nicefrac{l}{b}+1}}\Gamma\left( \frac{l}{b}+1,k+1 \right),
    \label{eq:DedutionStep5}
  \end{equation*}
  where $\Gamma(a,x)=\int_x^\infty y^{a-1}\e^{-y}\,dy$ denotes the complementary incomplete gamma function.
  Then, the $r$th ordinary moment of $X$ is given by
  \begin{equation}
    \mu_r'=\frac{\alpha\beta}{a^r}\sum_{j,k=0}^\infty\sum_{l=0}^r(-1)^{k+r+l}\,t_j\,\binom{j}{k}\binom{r}{l}\frac{\e^{k+1}}{(k+1)^{\nicefrac{l}{b}+1}}\Gamma\left( \frac{l}{b}+1,k+1 \right).
    \label{eq:OrdinaryOfEGNH}
  \end{equation}

  For $\alpha>0$ and $\beta>0$ integers, the moments in \eqref{eq:OrdinaryOfEGNH} reduce to
  \begin{equation*}
    \mu_r'=\frac{\alpha\beta}{a^r}\sum_{j=0}^{\alpha(i+1)-1}\sum_{k=0}^\infty\sum_{l=0}^r (-1)^{k+r+l}\,t_j\,\binom{j}{k}\binom{r}{l}\frac{\e^{k+1}}{(k+1)^{\nicefrac{l}{b}+1}}\Gamma\left( \frac{l}{b}+1,k+1 \right),
    \label{eq:OrdinaryOfEGNHtoNH}
  \end{equation*}
  where $t_j$ becomes
  \begin{equation*}
    t_j^*=\sum_{i=0}^{\beta-1}(-1)^{i+j}\binom{\beta-1}{i}\binom{\alpha(i+1)-1}{j}.
    \label{eq:pdfCoefficientsExpansion-alternative}
  \end{equation*}
  Additionally, the expression for $\mu_r'$ when $\alpha=\beta=1$ takes the form
  \begin{equation*}\label{eq:OrdinaryOfNH}
    \mu_r'=\frac{\e}{a^r} \sum_{l=0}^{r}(-1)^{l+r} \binom{r}{l}\Gamma\left(\frac{l}{b}+1,1\right),
  \end{equation*}
  which agrees with the moments given by \citet{NadarajahHaghighi-extensionexponentialdistribution-2011}.

  The central moments ($\mu_r$) of $X$ are given by
  \begin{equation}\label{eq:CentralandCumulants}
    \mu_r=\sum_{m=0}^r (-1)^m\binom{r}{m}\mu_1^{\prime m }\mu_{r-m}'.
  \end{equation}

  The cumulants ($\kappa_r$), skewness ($\gamma_1$) and kurtosis ($\gamma_2$) of $X$ are determined from \eqref{eq:OrdinaryOfEGNH} using well-known relationships $\kappa_r=\mu_r'-\sum_{m=0}^{r-1}\binom{r-1}{m-1}\kappa_m \mu_{r-m}'$, for $ r \ge 1$, $\gamma_1=\nicefrac{\kappa_3}{\kappa_2^{\nicefrac{3}{2}}}$ and $\gamma_2=\nicefrac{\kappa_4}{\kappa_2^{2}}$, where $\kappa_1=\mu_1'$.

  \section{Generating function} \label{sec:GeneratingFunction}

  The mgf of $X$ can be obtained from \eqref{eq:pdfExpansionAlternative} as
  \begin{align*}\label{eq:mgf2}
    M(s)=\sum_{j=0}^{\infty}u_j\,M_{j+1}(s),
  \end{align*}
  where $M_{j+1}(s)$ is the mgf of the ENH$(j+1,a,b)$ distribution \citep{BourguignonCarmoLeaoNascimentoPinhoCordeiro-newgeneralizedgamma-2015}
  given by
  \begin{equation}\label{eq:momentsENH}
    M_{j+1}(t)= \sum_{s,r=0}^{\infty} \frac{\eta_s\,g_{s,r}\,t^{j+1}}{r/b+1},
  \end{equation}
  where $\eta_i =\sum_{k=0}^{\infty}\frac{(-1)^i}{\lambda^kk!} \binom{k}{i}$, $g_{i,r}=(r\zeta_0)^{-1} \sum_{n=1}^{r}[n(i+1)-r]\zeta_n\,g_{i,r-n}$
  for $r \ge 1$, $g_{i,0}=\zeta_0^i$,  $\zeta_r = \sum_{m=0}^{\infty}f_m\,d_{m,r}$, $f_m=\sum_{j=m}^{\infty}(-1)^{j-m} \binom{j}{m} (\alpha^{-1})_j/j!$ and $d_{m,r}= (r a_0)^{-1}\sum_{v=0}^{r}[v(m+1)-r]\,a_v\,d_{m,r-v}$.
  Here, $(\alpha^{-1})_j=(\alpha^{-1})\times(\alpha^{-1}-1)\ldots(\alpha^{-1}-j+1)$ is the descending factorial.





  \section{Mean deviations} \label{sec:MeanDeviations}

  The mean deviations about the mean and the median of $X$ can be expressed, respectively, as $\delta_1(X)  =\E(|X-\mu_1'|)=2\mu_1'F(\mu_1')-2m_1(\mu_1')$ and $\delta_2(X) =\E(|X-\widetilde{X}|)=\mu_1'-2m_1(\widetilde{X})$, where $\mu_1'$ comes from \eqref{eq:OrdinaryOfEGNH}, $F(\mu_1')$ is easily calculated from \eqref{eq:cdfmodelo} and $m_1(z)= \int_{-\infty}^{z}\!xf(x)\,\mathrm{d}x$ is the first incomplete moment.

  A general equation for $m_1(z)$ can be derived from \eqref{eq:pdfExpansion} as
  \begin{equation}\label{eq:FirstIncompleteMoment}
    m_1(z)= \sum_{j=0}^{\infty}u_j\,J_{j+1}(z),
  \end{equation}
  where $J_{j+1}(z)= \int_{-\infty}^{z}\!x\,h_{j+1}(x)\,\mathrm{d}x$ and $h_{j+1}(x)$ is given by \eqref{eq:pdfExponentiatedNH}.
  After some algebra, we can prove that
  \begin{equation*}\label{eq:MeanDeviationsQuantity-expansion}
    J_{j+1}(z)= \frac{(j+1)}{a}\sum_{k=0}^{\infty} \sum_{l=0}^{1}\frac{(-1)^{k+l+1}\e^{k+1}}{(k+1)^{1+\frac{l+1}{b}}}\,\binom{j}{k}\,\phi(a,b,k,l;z),
  \end{equation*}
  where
  \begin{equation*}\label{eq:MeanDeviationsQuantity-expansion2}
    \phi(a,b,k,l;z)=\gamma\left( \frac{b+l+1}{b},(k+1)(1+az)^b\right) - \gamma\left( \frac{b+l+1}{b},k+1\right).
  \end{equation*}

  The Bonferroni and Lorenz curves are defined (for a given probability $\pi$) by $B(\pi)= \nicefrac{m_1(q)}{\pi\mu_1'}$ and $L(\pi)=\nicefrac{m_1(q)}{\mu_1'}$, respectively, where $q=Q(\pi)$ comes from \eqref{eq:quantile function}.
  Using \eqref{eq:FirstIncompleteMoment}, we can easily obtain $B(\pi)$ and $L(\pi)$.

  \section{Rényi Entropy} \label{sec:RenyiEntropy}

  The entropy of a random variable $X$ with pdf $f(x)$ is a measure of variation of the uncertainty.
  For any real parameter $\lambda>0$ and $\lambda\ne1$, the Rényi entropy is given by
  \begin{equation}\label{eq:RenyiEntropyDefinition}
    I_R(\lambda)= \frac{1}{(1-\lambda)} \log\left( \int_{-\infty}^{\infty}\!f(x)^\lambda\,\mathrm{d}x\right).
  \end{equation}

  Inserting \eqref{eq:cdfbaseline} in equation \eqref{eq:RenyiEntropyDefinition} and, after some algebra, we obtain
  \begin{equation*}\label{eq:RenyiEntropyPasso1}
    I_R(\lambda)= \frac{1}{1-\lambda}\left\{\lambda\log(ab\alpha\beta)+\alpha\lambda+\log\int_{0}^{\infty}\! \frac{(1+ax)^{\lambda(b-1)}\exp\left[-\alpha\lambda\left(1+ax\right)^b\right]}{\left\{1-\exp\left[1-\left(1+ax\right)^b\right]^\alpha\right\}^{\lambda(1-\beta)}}\,\mathrm{d}x\right\}.
  \end{equation*}

  Applying the binomial expansion gives
  \begin{equation*}\label{eq:RenyiEntropyPasso2}
    I_R(\lambda)= -\log(ab)+ \frac{1}{1-\lambda}\left\{\lambda\log(\alpha\beta)+\alpha\lambda+\log \Sigma(\alpha,\beta,b,\lambda)\right\},
  \end{equation*}
  where
  \begin{equation}\label{eq:RenyiEntropyPasso2-quantity}
    \Sigma(\alpha,\beta,b,\lambda) =\sum_{i=0}^{\infty} \binom{\lambda(\beta-1)}{i}\frac{(-1)^i \e^{i\alpha}}{\left[(\lambda+i)\alpha\right]^{ \frac{\lambda(b-1)+1}{b}}} \Gamma\left( \frac{\lambda(b-1)+1}{b},(\lambda+i)\alpha\right).
  \end{equation}

  For $\beta>0$ integer, Equation \eqref{eq:RenyiEntropyPasso2-quantity} takes the form
  \begin{equation*}\label{eq:RenyiEntropyPasso3-quantity}
    \Sigma(\alpha,\beta,b,\lambda) =\sum_{i=0}^{\beta-1} \binom{\lambda(\beta-1)}{i}\frac{(-1)^i \e^{i\alpha}}{\left[(\lambda+i)\alpha\right]^{ \frac{\lambda(b-1)+1}{b}}} \Gamma\left( \frac{\lambda(b-1)+1}{b},(\lambda+i)\alpha\right).
  \end{equation*}

  For $\alpha=\beta=1$, the expression above reduces to
  \begin{equation*}\label{eq:RenyiEntropyNH}
    I_R(\lambda)= -\log(ab)+ \frac{1}{1-\lambda}\left\{\lambda- \frac{\lambda(b-1)+1}{b}\log(\lambda)+\log\left[\Gamma\left( \frac{\lambda(b-1)+1}{b},\lambda\right)\right]\right\},
  \end{equation*}
  which agrees with a result given by \cite{NadarajahHaghighi-extensionexponentialdistribution-2011}.

  Figure \ref{fig:RenyiEntropy} displays plots of some Rényi entropy curves versus $\lambda$ for the EGNH, ENH and NH distributions with some parameter values. 
  In this case, one can note that the NH entropy is fixed thus, we can measure the impact of the change of $\alpha$ and $\beta$ additional parameters on the EGNH entropy. This quantity increases when $\alpha$ and $\beta$ increases with respect the comparison between ENH and EGNH entropies, one can observe that the former is a lower bond of the second and the distance between them is higher when $\alpha$ increases.
  
  \begin{figure*}[htb!] 
    \centering
    \subfigure[$\alpha=\beta=3,a=b=0.5$]{
      \includegraphics[width=0.31\textwidth]{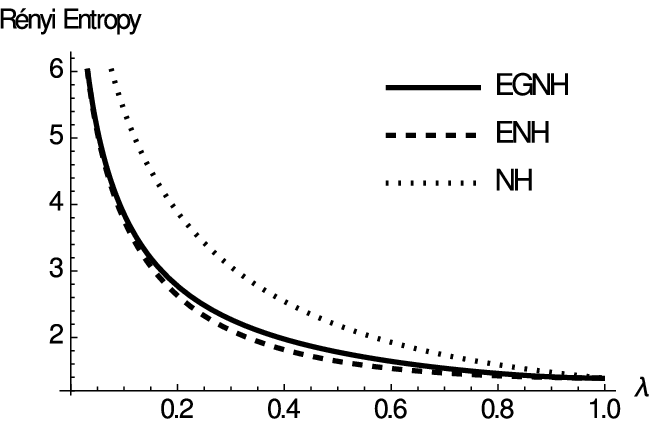}
      \label{fig:RenyiEntropy1}
    }
    \subfigure[$\alpha=\beta=5,a=b=0.5$]{
      \includegraphics[width=0.31\textwidth]{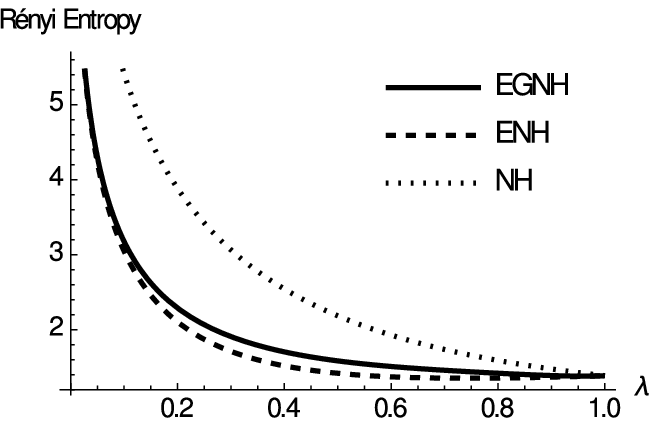}
      \label{fig:RenyiEntropy2}
    }
    \subfigure[$\alpha=\beta=7,a=b=0.5$]{
      \includegraphics[width=0.31\textwidth]{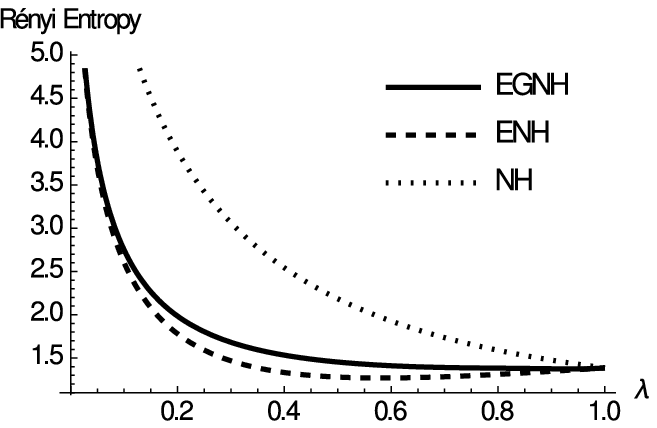}
      \label{fig:RenyiEntropy3}
    }
    \caption{Rényi entropies curves versus $\lambda$ for the EGNH, ENH and NH distributions}
    \label{fig:RenyiEntropy}
  \end{figure*}




  \section{Order statistics} \label{sec:OrderStatistics}

  The pdf $f_{i:n}(x)$ of the $i$th order statistic, say $X_{i:n}$, for $i=1,\ldots,n$, from independent identically distributed random variables $X_1,\ldots X_n$ is given by
  $$f_{i:n}(x)= \frac{f(x)}{B(i,n-i+1)}F(x)^{i-1}\left[1-F(x)\right]^{n-i}.$$

  Following a result given by \cite{CordeiroOrtegaCunha-ExponentiatedGeneralizedClass-2013}, the 
  pdf $f_{i:n}(x)$ for the EGNH distribution can be expressed as 
  \begin{equation}\label{eq:OrderStatisticsEGclass}
    f_{i:n}(x)= \frac{\alpha\beta}{B(i,n-i)} \sum_{l=0}^{\infty}\,s_l\,h_{l+1}(x),
  \end{equation}
  where
  \begin{equation*}\label{eq:OrderStatisticsCoeficients}
    s_l= \sum_{j=0}^{n-i} \sum_{k=0}^{\infty} (-1)^{j+k+l} \binom{n-i}{j} \binom{\beta(j+1)-1}{k} \binom{\alpha(k+1)-1}{l}
  \end{equation*}
  and $h_{l+1}(x)$ is the ENH$(l+1,a,b)$ pdf given by \eqref{eq:pdfExponentiatedNH}.

  Equation \eqref{eq:OrderStatisticsEGclass} reveals that the pdf of the EGNH order statistic is a linear combination of ENH densities.
  A direct application of this expansion allows us to determine the ordinary moments and the mgf of the EGNH order statistics.

  The $r$th ordinary moment of $X_{i:n}$ is given by
  \begin{equation}\label{eq:OrderDefinition}
    \E(X^r_{i:n})= \frac{\alpha\beta}{B(i,n-i+1)} \sum_{l=0}^{\infty}s_l\,\int_{0}^{\infty}\!x^r\,h_{l+1}(x)\,\mathrm{d}x.
  \end{equation}
  Equation \eqref{eq:OrderDefinition} reveals that the moments of the order statistics depend on the moments of the ENH$(l+1,a,b)$ distribution. \citet{Lemonte-newexponentialtype-2013} determined the moments of the ENH distribution as
  \begin{equation*}\label{eq:OrdinaryOfENH}
    \mu_r'= \frac{\beta}{a^r} \sum_{m=0}^{r} \sum_{p=0}^{\infty}\binom{r}{m} \binom{\beta-1}{p} \frac{(-1)^{r-m+p}\e^{p+1}}{(p+1)^{\frac{m}{b}+1}} \Gamma\left( \frac{m}{b}+1,p+1\right).
  \end{equation*}

  From this result, Equation \eqref{eq:OrderDefinition} can be reduced to
  \begin{equation*}\label{eq:OrderEGNH}
    \E(X^r_{i:n})= \frac{\alpha\beta^2}{B(i,n-i+1)} \sum_{l=0}^{\infty} s^*_l\,\Gamma\left( \frac{m}{b}+1,t+1\right),
  \end{equation*}
  where
  \begin{equation*}\label{eq:OrderStatisticsCoeficients2}
    s^*_l =\sum_{m=0}^{r} \sum_{p=0}^{\infty} s_l\,\binom{r}{m} \binom{\beta-1}{p} \frac{(-1)^{r-m+p}\e^{p+1}}{a^r(p+1)^{ \frac{m}{b}+1}}.
  \end{equation*}

  Further, the mgf of $X_{i:n}$ can follow from \eqref{eq:OrderStatisticsEGclass} as
  \begin{equation*}\label{eq:OrderEGNHmgf}
    M(t)= \frac{\alpha\beta}{B(i,n-i+1)} \sum_{l=0}^{\infty}s_l\,\int_{0}^{\infty}\e^{tx}\,h_{l+1}(x)\,\mathrm{d}x = \frac{\alpha\beta}{B(i,n-i+1)} \sum_{l=0}^{\infty}s_l\,M_{l+1}(t),
  \end{equation*}
  where $M_{l+1}(t)$ is given by \eqref{eq:momentsENH}.


  \section{Maximum likelihood estimation} \label{sec:MLE}

  In this section, we provide an analytic procedure to obtain the MLEs for the EGNH parameters using the \emph{profile log-likelihood function} (pllf).
  The MLEs are employed due to their interesting asymptotic properties.
  By means of such properties, one can construct confidence intervals for the model parameters.

  Let X be the EGNH distribution with vector of parameters $ \boldsymbol{\theta}=(\alpha,\beta,a,b)^\top $.
  The \emph{log-likelihood function} for $ \boldsymbol{\theta} $, given the observed sample $ x_1,\ldots,x_n $, due to \eqref{eq:pdfmodelo} is given by
  \begin{equation}\label{eq:log-likelihood function}
    \begin{split}
      \ell(\boldsymbol{\theta})=&n[\alpha+\log(a\alpha b\beta)]+(b-1)\sum_{i=1}^{n}\log(ax_i+1)-\alpha\sum_{i=1}^{n}(ax_i+1)^b\\&+
      (\beta-1)\sum_{i=1}^{n}\log\left\{1-\left[\e^{1-(ax_i+1)^b}\right]^\alpha\right\}.
    \end{split}
  \end{equation}

  The components of the associated score vector $ \textbf{U}=\left(\frac{\partial\ell}{\partial\alpha},\frac{\partial\ell}{\partial\beta},\frac{\partial\ell}{\partial a},\frac{\partial\ell}{\partial b}\right)^\top $ are
  \begin{flalign*}
    \frac{\partial\ell(\boldsymbol{\theta})}{\partial\alpha}=&\frac{n}{\alpha} + \sum\limits_{i=1}^{n}\left[1-(1+ax_i)^b\right]\left\{1- \frac{(\beta-1) \{ \exp\left[1-(1+ax_i)^b\right] \}^\alpha}{1-\left\{ \exp\left[1-(1+ax_i)^b\right] \right\}^\alpha} \right\} ,&&
  \end{flalign*}
  \begin{flalign}
    \frac{\partial\ell(\boldsymbol{\theta})}{\partial\beta}=&\frac{n}{\beta} + \sum\limits_{i=1}^{n}\log\left\{ 1-\left\{ \exp\left[1-(1+ax_i)^b\right] \right\}^\alpha \right\} \label{eq:scorebeta},&&
  \end{flalign}
  \begin{flalign*}
    \frac{\partial\ell(\boldsymbol{\theta})}{\partial a}=&\frac{n}{a} + (b-1)\sum\limits_{i=1}^{n}\frac{x_i}{1+ax_i} - \alpha\sum\limits_{i=1}^{n}\left[bx_i(1+ax_i)^{b-1}\right]\nonumber&&\\
    &+\alpha(\beta-1)\sum\limits_{i=1}^{n}\frac{bx_i(1+ax_i)^{b-1}\exp\left\{ \alpha\left[ 1-(1+ax_i)^b \right] \right\}}{1-\{ \exp\left[1-(1+ax_i)^b\right] \}^\alpha},&&
  \end{flalign*}
  and
  \begin{flalign*}
    \frac{\partial\ell(\boldsymbol{\theta})}{\partial b}=&\frac{n}{b} + \sum\limits_{i=1}^{n}\log(1+ax_i) - \alpha\sum\limits_{i=1}^{n}\left\{\left[\log(1+ax_i)\right](1+ax_i)^b\right\}\nonumber&&\\
    &+\alpha(\beta-1)\sum\limits_{i=1}^{n}\frac{ \left[\log(1+ax_i)\right](1+ax_i)^b\exp\left\{ \alpha\left[ 1-(1+ax_i)^b \right] \right\}}{1-\{ \exp\left[1-(1+ax_i)^b\right] \}^\alpha}.&&
  \end{flalign*}

  The MLEs, $\boldsymbol{\hat{\theta}}=(\hat{\alpha},\hat{\beta},\hat{a},\hat{b})^\top$, can be obtained numerically by solving the system of nonlinear equations
  \begin{equation*}
    \frac{\partial\ell(\boldsymbol{\theta})}{\partial\alpha} = \frac{\partial\ell(\boldsymbol{\theta})}{\partial\beta} = \frac{\partial\ell(\boldsymbol{\theta})}{\partial a} = \frac{\partial\ell(\boldsymbol{\theta})}{\partial b} = 0.
  \end{equation*}

  Note that from \eqref{eq:scorebeta} it is possible to obtain a semi-closed estimator of $\beta$.
  From $\frac{\partial\ell(\boldsymbol{\theta})}{\partial\beta}=0$, the estimator for $\beta$ is given by
  \begin{equation}
    \hat{\beta}(\hat{\alpha},\hat{a},\hat{b})=-\frac{n}{\sum\limits_{i=1}^{n}\log\left\{ 1-\left\{ \exp\left[1-(1+\hat{a}x_i)^{\hat{b}}\right] \right\}^{\hat{\alpha}}\right\}}.
    \label{eq:betaestimator}
  \end{equation}
  
  Based on \eqref{eq:betaestimator}, fixed on $x_1,\ldots,x_n$, $\hat{\beta}\to\hat{0}^+$ when $\hat{b}\to\hat{0}^+$ and/or $\hat{\alpha}\to\hat{0}^+$.
  This behavior anticipates that estimates for smaller $\alpha$ and/or $b$ may require improved estimation procedures.

  Another way to obtain the MLEs is using the pllf.
  We suppose $ \beta $ fixed and rewrite \eqref{eq:log-likelihood function} as $ \ell(\boldsymbol{\theta})=\ell_\beta(\alpha,a,b)$ (to show that $ \beta $ is fixed but $\boldsymbol{\theta}_\beta=(\alpha, a, b )^\top$ varies).
  To estimate $ \boldsymbol{\theta}_\beta $, we maximize $ \ell_\beta(\alpha,a,b) $ with respect to $ \boldsymbol{\theta}_\beta $, i.e.,
  \begin{equation*}\label{eq:profile estimator theta beta}
    \boldsymbol{\hat{\theta}}_\beta=\arg\max_{\boldsymbol{\theta}_\beta}\ell_\beta(\alpha,a,b).
  \end{equation*}

%

  By replacing $\beta$ by $\hat{\beta}(\alpha,a,b)$ in \eqref{eq:log-likelihood function}, we can obtain the pllf for $\alpha$, $a$ and $b$ as

  \begin{align}\label{eq:ProfileLogLikelihoodFuntion}
    \ell_{\beta}(\alpha,a,b) = & \left\{ \alpha+\log  \left\{ \alpha-{\frac {nab}{\sum _{i=1}^{n}\log \left[ 1- \left( {{\rm e}^{1- \left( ax_{{i}}+1 \right) ^{b}}} \right) ^{\alpha} \right] }} \right\}  \right\}                                                                       \nonumber\\
    & + \left( b-1 \right) \sum _{i=1}^{n}\log  \left( ax_{{i}}+1 \right) -\alpha\,\sum _{i=1}^{n} \left( ax_{{i}}+1 \right) ^{b}                                                                                                                          \nonumber\\
    & - \left\{ {\frac {n}{\sum _{i=1}^{n} \log  \left[ 1- \left( {{\rm e}^{1- \left( ax_{{i}}+1 \right) ^{b}}} \right) ^{\alpha} \right] }}+1 \right\} \sum _{i=1}^{n}\log  \left\{ 1- \left[ {{\rm e}^{1- \left( ax_{{i}}+1 \right) ^{b}}} \right] ^{\alpha
      } \right\}.
  \end{align}

  As a consequence MLEs for $\alpha$, $a$ and $b$ can also be determined by maximizing the pllf \eqref{eq:ProfileLogLikelihoodFuntion} with respect to these parameters.
  Another way to obtain the MLEs is to solve the \emph{profile log-likelihood equations}
  \begin{align*}\label{eq:ProfileLogLikelihoodEquations}
    \frac{\partial\ell_\beta(\alpha,a,b)}{\partial\alpha} = \frac{\partial\ell_\beta(\alpha,a,b)}{\partial a} =\frac{\partial\ell_\beta(\alpha,a,b)}{\partial b}=0.
  \end{align*}

  \noindent It should be noted that the maximization of \eqref{eq:ProfileLogLikelihoodFuntion} is simpler than the maximization of \eqref{eq:log-likelihood function} with respect to their respective parameters, since the pllf has three parameters, whereas the llf has four parameters.

  For interval estimation on the model parameters, we require the $ 4\times 4 $ observed information matrix, $J(\theta)$, whose elements.

  Under standard regularity conditions, the asymptotic distribution of $\sqrt{n}(\boldsymbol{\hat{\theta}}-\boldsymbol{\theta})$ is multivariate normal $ N_4(0,I(\boldsymbol{\theta})^{-1}) $, where $ I(\boldsymbol{\theta}) $ is the expected information matrix.
  This approximated distribution holds when $ I(\boldsymbol{\theta}) $ is replaced by $ J(\boldsymbol{\theta}) $.
  The multivariate normal $ N_4(0,J(\boldsymbol{\theta})^{-1}) $ distribution for $\sqrt{n}(\boldsymbol{\hat{\theta}}-\boldsymbol{\theta})$ can be used to construct approximate confidence intervals for the individual parameters.

  \section{Simulation studies} \label{sec:SimulationStudies}

  In this section, we assess the performance of the MLEs for $\alpha$, $\beta$, $a$ and $b$  by means of a Monte Carlo study.
  To that end, we adopt sample sizes $n\in\{10,15,\ldots,250\}$ and the parameter values $\boldsymbol{\theta}=(\alpha,\beta,a,b)=(1.8\times10^{-3},2.83\times10^{-1},1.75\times10^{-3},47.066)$ (defined from the first application to real data discussed in Section \ref{sec:Applications}).

  Our simulation study consists of the following steps.
  For each pair $(\boldsymbol{\theta},n)$, with $\boldsymbol{\theta}=(\alpha,\beta,a,b)$:
  \begin{enumerate}
    \item generate $n$ outcomes of the EGNH distribution;
    \item obtain the MLEs;
    \item compute the average biases and standard errors.
  \end{enumerate}
  The simulation study is performed using the \texttt{R} programming language and the BFGS (constrained) method. 
  For maximizing the pllf \eqref{eq:ProfileLogLikelihoodFuntion}, we use the MaxLik package \citep{HenningsenToomet-maxLik:packagemaximum-2011}.

  Figures \ref{fig:ConvergencePlot1-stderr} and \ref{fig:ConvergencePlot1-bias} display results of this simulation study.
  As expected, both biases and standard errors decrease when the sample size increases.
  This proposed estimation procedure can be used effectively for estimating the parameters of the EGNH family. Further, it is important to note that the estimate for $\beta$ are good, even before the small value considered of $\alpha$.

  \begin{figure}[!htb]
    \includegraphics[width=\textwidth]{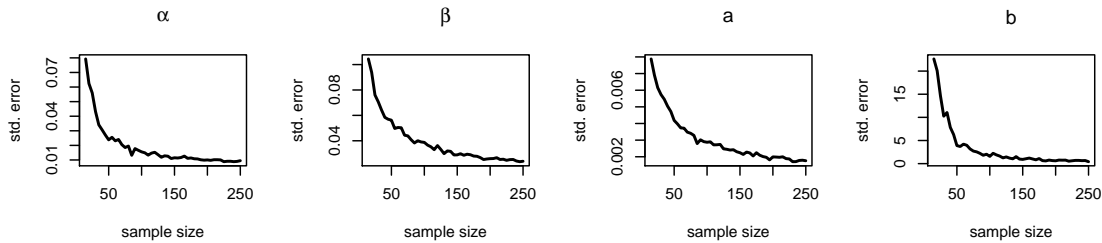}
    \caption{Std. errors of the estimates vs. sample size for $(\alpha,\beta,a,b)=(1.8\times10^{-3},2.83\times10^{-1},1.75\times10^{-3},47.066)$}
    \label{fig:ConvergencePlot1-stderr}
    \centering
  \end{figure}

  \begin{figure}[!htb]
    \includegraphics[width=\textwidth]{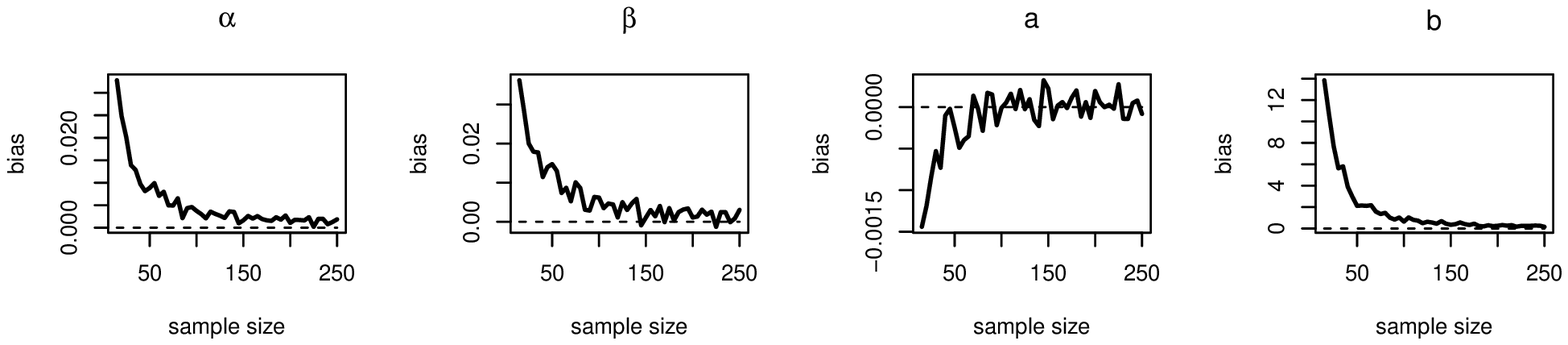}
    \caption{Biases of the estimates vs. sample size for $(\alpha,\beta,a,b)=(1.8\times10^{-3},2.83\times10^{-1},1.75\times10^{-3},47.066)$}
    \label{fig:ConvergencePlot1-bias}
    \centering
  \end{figure}



  \section{Applications} \label{sec:Applications}

  \begin{figure}[htb!]
    \centering
    \subfigure[\citeauthor{Aarset-HowtoIdentify-1987}]{
      \includegraphics[width=0.45\textwidth]{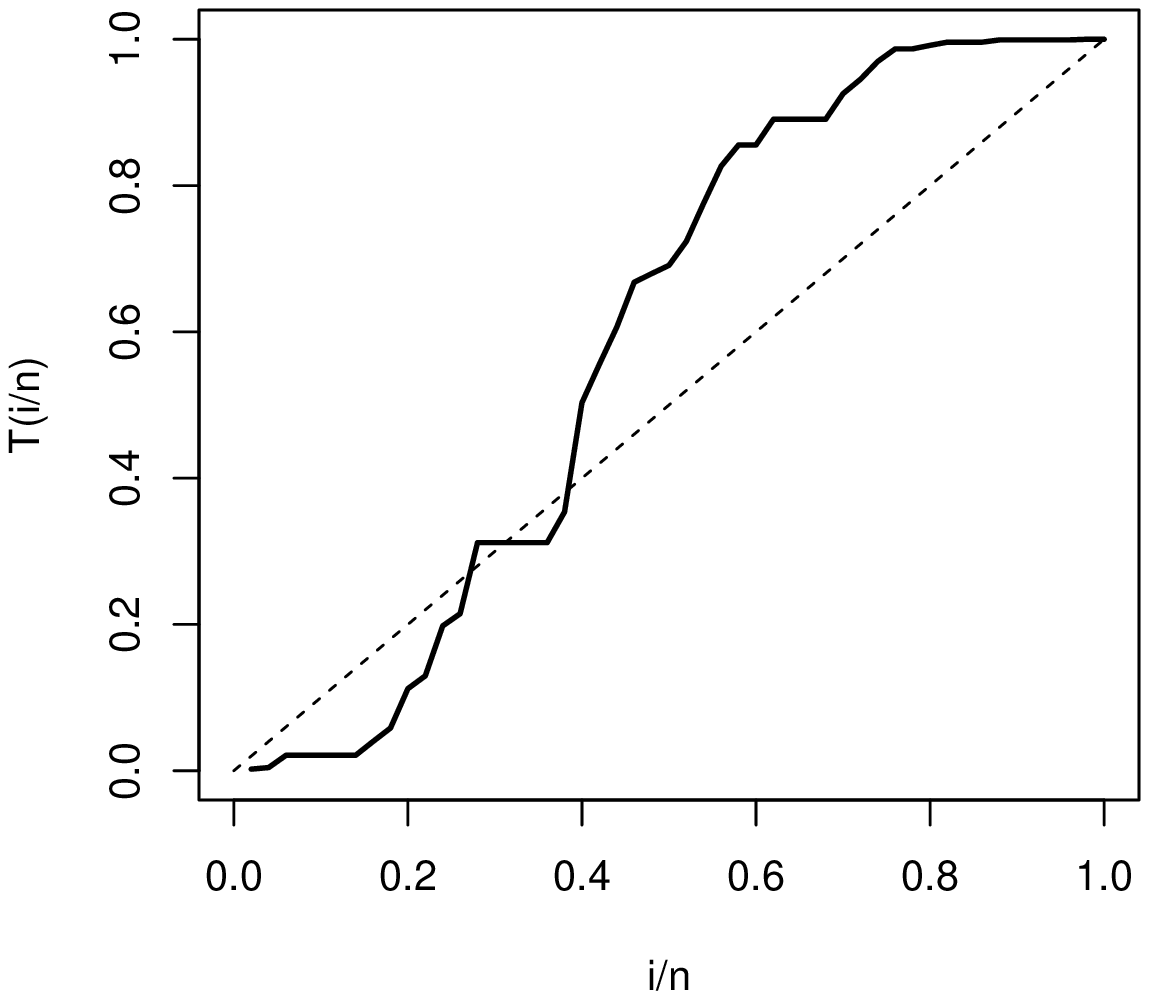}
      \label{fig:TTT-Application1}
    }
    \subfigure[\citeauthor{AndrewsHerzberg-StressRuptureLife-1985}]{
      \includegraphics[width=0.45\textwidth]{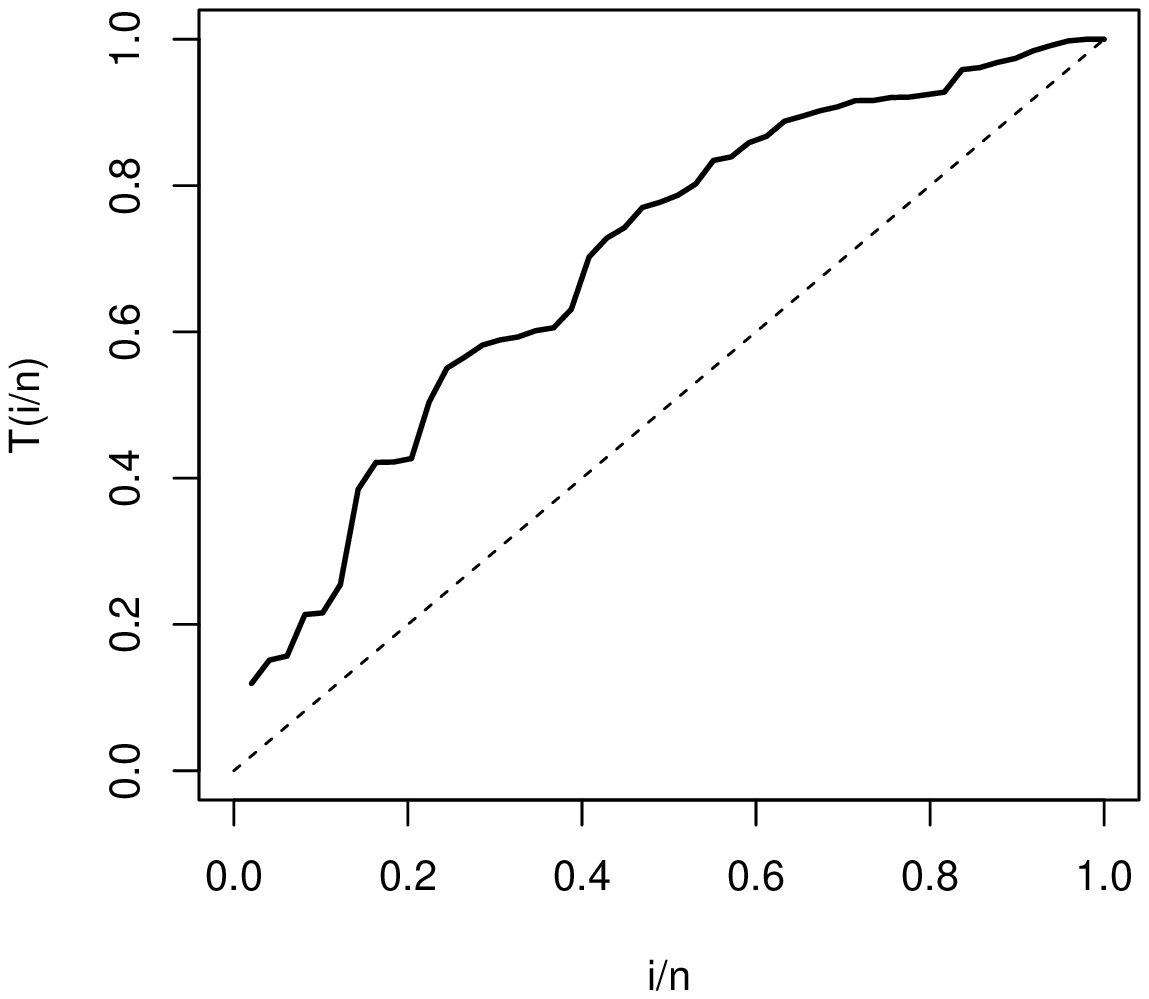}
      \label{fig:TTT-Application2}
    }
    \caption{TTT-plots of datasets}
    \label{fig:TTT-Applications}
  \end{figure}

  In this section, we present two applications to real data in order to illustrate the EGNH potentiality.
  The first uncensored real dataset refers to the lifetimes of 50 components \citep{Aarset-HowtoIdentify-1987}: 0.1, 7, 36, 67, 84, 0.2, 11, 40, 67, 84, 1, 12, 45, 67, 84, 1, 18, 46, 67, 85, 1, 18, 47, 72, 85, 1, 18, 50, 75, 85, 1, 18, 55, 79, 85, 2, 18, 60, 82, 85, 3, 21, 63, 82, 86, 6, 32, 63, 83, 86. The second uncensored dataset represents the stress-rupture life of kevlar 49/epoxy strands, which were subjected to constant sustained pressure at 70\% stress level until all units had failed \citep{AndrewsHerzberg-StressRuptureLife-1985}: 1051, 1337, 1389, 1921, 1942, 2322, 3629, 4006, 4012, 4063, 4921, 5445, 5620, 5817, 5905, 5956, 6068, 6121, 6473, 7501, 7886, 8108, 8546, 8666, 8831, 9106, 9711, 9806, 10205, 10396, 10861, 11026, 11214, 11362, 11604, 11608, 11745, 11762, 11895, 12044, 13520, 13670, 14110, 14496, 15395, 16179, 17092, 17568, 17568.
  For a previous study with this dataset, see \citet{CoorayAnanda-GeneralizationHalfNormal-2008}.

  Table \ref{table:DescriptiveStatistics} gives a descriptive summary of the two samples.
  For both datasets, the median is greater than the mean (indicating asymmetry).
  The \citeauthor{Aarset-HowtoIdentify-1987} dataset has negative skewness (left-skewed data), while the \citeauthor{AndrewsHerzberg-StressRuptureLife-1985} dataset has positive one (right-skewed data).

  \begin{table}[htb!]
    \caption{Descriptive Statistics}
    \begin{tabular}{lrrrrrrr}
      \hline
      Datasets                                   & Mean     & Median & Variance & Skew. & Kurt. & Min. & Max.\\ \hline
      \citeauthor{Aarset-HowtoIdentify-1987}                  & 45.686   & 48.5   & 1078.153 & -0.1378  & 1.414 & 0.1     & 86 \\
      \citeauthor{AndrewsHerzberg-StressRuptureLife-1985} & 8805.694 & 8831   & 20738145 & 0.097    & 2.172    & 1051    & 17568 \\
      \hline
    \end{tabular}
    \label{table:DescriptiveStatistics}
  \end{table}

  In many applications empirical hrf shape can indicate a particular model.
  The plot of total time on test (TTT) \citep{Aarset-HowtoIdentify-1987} can be useful in this sense.
  The TTT plot is obtained by plotting $G(r/n)= \nicefrac{( \sum_{i=1}^{r}y_{i:n}+(n-r)y_{r:n})}{\sum_{i=1}^{n}y_{i:n}}$, where $r=1,\ldots,n$ and $y_{i:n} (i=1,\ldots,n)$ are the order statistics of the sample, against $r/n$.
  The TTT plots for the current datasets are presented in Figure \ref{fig:TTT-Applications}.
  The TTT plot for the \citeauthor{Aarset-HowtoIdentify-1987}'s data indicates a bathtub-shaped hrf and the TTT plot for the \citeauthor{AndrewsHerzberg-StressRuptureLife-1985}'s data indicates an increasing hrf.
  Therefore, these plots reveal the appropriateness of the EGNH distribution to fit these datasets, since the new model can present bathtub-shaped and increasing hrfs.

  We compare the fits of the EGNH distribution defined in \eqref{eq:pdfmodelo} with the submodels NH, EE and ENH and additionally with some other lifetime models with four, three and two parameters, namely:
  \begin{itemize}
    \item The Kumaraswamy Nadarajah-Haghighi (KNH) distribution \citep{Lima-TheHalf-normalGeneralizedFamilyAndKumaraswamyNadarajah-haghighiDistribution-2015} having pdf (for $x>0$) given by
      \begin{align*}
        f_{KNH}(x)= \frac{ab\alpha\beta(1+ax)^b\exp\left[1-(1+ax)^b\right]\left\{1-\exp\left[1-(1+ax)^b\right]\right\}^{\alpha-1}}{\left\{1-\left\{1-\exp\left[1-\left(1+ax\right)^b\right]\right\}^\alpha\right\}^{1-\beta}}.
      \end{align*}
    \item The beta Nadarajah-Haghighi (BNH) distribution \citep{Dias-NewContinuousDistributionsAppliedToLifetimeDataAndSurvivalAnalysis-UniversidadeFederaldePernambuco-2016} having pdf (for $x>0$) given by
      \begin{align*}
        f_{BNH}(x)=\frac{ab}{B(\alpha,\beta)}(1+ax)^{b-1}\left\{1-\exp\left[1-(1+ax)^b\right]\right\}^{\alpha-1}\left\{\exp\left[1-(1+ax)^b\right]\right\}^{\beta}.
      \end{align*}
    \item The Zografos-Balakrishnan Nadarajah-Haghighi (ZBNH) distribution \citep{BourguignonCarmoLeaoNascimentoPinhoCordeiro-newgeneralizedgamma-2015} having pdf (for $x>0$) given by
      \begin{align*}
        f_{ZBNH}(x)= \frac{ab}{\Gamma(\alpha)}(1+ax)^b\left[\left(1+ax\right)^b-1\right]^{\alpha-1}\exp\left\{-\left[\left(1+ax\right)^b-1\right]\right\}.
      \end{align*}
    \item The Marshall-Olkin Nadarajah-Haghighi (MONH) distribution \citep{LemonteCordeiroMoreno-Arenas-newusefulthree-2016} having pdf (for $x>0$) given by
      \begin{align*}
        f_{MONH}(x)=  \frac{ab\alpha(1+ax)^{b-1}\exp\left[1-\left(1+ax\right)^b\right]}{\left\{1-\left(1-\alpha\right)\exp\left[1-\left(1+ax\right)^b\right]\right\}^2}.
      \end{align*}
    \item The transmuted Nadarajah-Haghighi (TNH) distribution \citep{AhmedMuhammedElbatal-NewClassExtension-2015} having pdf (for $x>0$) given by
      \begin{align*}
        f_{TNH}(x)=ab (1+a x)^{b-1}\exp\left[1-\left(1+a x\right)^b\right]\left\{\left(1-\alpha\right)+2\alpha\exp\left[1-\left(1+a x\right)^b\right]\right\}.
      \end{align*}
    \item The modified Nadarajah-Haghighi (MNH) distribution \citep{El-DamceseRamadan-StudiesPropertiesand-2015} having pdf (for $x>0$) given by
      \begin{align*}
        f_{MNH}(x)=\alpha(\lambda+2\beta x)(1+\lambda x+\beta x^2)^{\alpha-1}\exp\left[1-\left(1+\lambda x+\beta x^2\right)^\alpha\right].
      \end{align*}
    \item The generalized power Weibull (GPW) distribution \citep{NikulinHaghighi-ChiSquaredTest-2006} having pdf (for $x>0$) given by
      \begin{align*}
        f_{GPW}(x) = abcx^{c-1}\left( 1+a{x}^{c} \right) ^{b-1}\exp[1- \left( 1+a{x }^{c} \right) ^{b}].
      \end{align*}
    \item The classic Weibull distribution having pdf (for $x>0$) given by
      \begin{align*}
        f_{Weibull}(x) = \left( \frac{a}{b}\right) \left( \frac{x}{b}\right)^{a-1} \exp\left[- \left( \frac{x}{b}\right)^a\right].
      \end{align*}
  \end{itemize}


    Tables \ref{table:EstimativesApplication1} and \ref{table:EstimativesApplication2} present the MLEs and their standard errors for some models fitted to the \citeauthor{Aarset-HowtoIdentify-1987}'s and \citeauthor{AndrewsHerzberg-StressRuptureLife-1985}'s datasets, respectively.
  In order to compare the competitive models, Tables \ref{table:GoodnessOfFitApplication1} and \ref{table:GoodnessOfFitApplication2} list the Goodness-of-Fit (GoF) statistics of the fitted models for the \citeauthor{Aarset-HowtoIdentify-1987}'s and \citeauthor{AndrewsHerzberg-StressRuptureLife-1985}'s datasets, respectively.
  We use the following GoF statistics:
  Cramér-von Mises (W*),
  Anderson Darling (A*),
  Kolmogorov-Smirnov (KS),
  Akaike Information Criterion (AIC),
  Bayesian Information Criterion (BIC),
  Consistent Akaike Information Criterion (CAIC) and
  Hannan-Quinn Information Criterion (HQIC).
  They are used to verify the quality of the fits.
  In general, smaller GoF values are associated with better fits.
  
  \begin{table}[htb!]
  	\centering
  	\caption{Estimates for \citeauthor{Aarset-HowtoIdentify-1987}'s dataset}
  	\begin{tabular}{lcccc}
  		\hline
  		model                     & estimates                               &                                          &                                         &                       \\ \hline
  		$\EGNH(\alpha,\beta,a,b)$ & $1.8\times10^{-3}$                      & $2.83\times10^{-1}$                      & $1.75\times10^{-3}$                     & 47.066                \\
  		& $(8.8\times10^{-4})$                    & $(1.19\times10^{-2})$                    & $(2.47\times10^{-4})$                   & $(3.69)$              \\
  		BNH$(\alpha,\beta,a,b)$   & $2.55\times10^{-1}$                     & $(5.2\times10^{-2})$                     & $5.225\times10^{-3}$                    & 9.753                 \\
  		& $(2.41\times10^{-2})$                   & $(1.93\times10^{-2})$                    & $(5.018\times10^{-4})$                  & $(3.20)$              \\
  		KNH$(\alpha,\beta,a,b)$   & $1.22\times10^{-1}$                     & $6.27\times10^{-2}$                      & $5.85\times10^{-4}$                     & 71.05                 \\
  		& $(3.10\times10^{-2})$                   & $(2.91\times10^{-2})$                    & $(3.98\times10^{-5})$                   & $(1.26\times10^{-1})$ \\
  		MNH$(\alpha,a,b)$         & $2.78\times10^{-6}$                     & $2.01\times10^{-4}$                      & 30.30                                   &                       \\
  		& $(3.98\times10^{-7})$                   & $(5.21\times10^{-5})$                    & $(8.61\times10^{-1})$                   &                       \\
  		MONH$(\alpha,a,b)$        & 4.36                                    & $4.61\times10^{-5}$                      & 370.406                                 &                       \\
  		& $(4.78\times10^{-1})$                   & $(4.92\times10^{-6})$                    & $(60.91)$                               &                       \\
  		TNH$(\alpha,a,b)$         & $(-4.03\times10^{-1})$                  & $4.43\times10^{-6}$                      & 3080.45                                 &                       \\
  		& $(1.01\times10^{-1})$                   & $(3.50\times10^{-6})$                    & $(112.90)$                              &                       \\
  		ENH$(\alpha,a,b)$         & $8.29\times10^{-1}$                     & $2.32\times10^{-5}$                      & 509.957                                 &                       \\
  		& $(7.03\times10^{-2})$                   & $(2.06\times10^{-6})$                    & (60.71)                                 &                       \\
  		ZBNH$(\alpha,a,b)$        & $8.91\times10^{-1}$                     & $3.41\times10^{-6}$                      & 3437.14                                 &                       \\
  		& $(7.55\times10^{-2})$                   & $(1.54\times10^{-7})$                    & (106.40)                                &                       \\
  		GPW$(a,b,c)$              & 395.387                                 & 1.013                                    & $(2.99\times10^{-5})$                   &                       \\
  		& (49.66)                                 & $(8.23\times10^{-2})$                    & $(1.31\times10^{-5})$                   &                       \\
  		NH$(a,b)$                 & $1.6\times10^{-5}$                      & 775.132                                  &                                         &                       \\
  		& $(2.10\times10^{-6})$                   & (97.86)                                  &                                         &                       \\
  		EE$(\alpha,\beta)$        & $8.85\times10^{-1}$                     & $1.70\times10^{-2}$                      &                                         &                       \\
  		& $(8.74\times10^{-2})$                   & $(9.53\times10^{-4})$                    &                                         &                       \\
  		Weibull$(a,b)$            & 1.13                                    & 56.08                                    &                                         &                       \\
  		& $(9.29\times10^{-2})$                   & (3.25)                                   &                                         &                       \\
  		\hline
  	\end{tabular}
  	\label{table:EstimativesApplication1}
  \end{table}
  
  \begin{table}[htb!]
  	\centering
  	\caption{Estimates for \citeauthor{AndrewsHerzberg-StressRuptureLife-1985}'s dataset}
  	\begin{tabular}{lcccc}
  		\hline
  		model                      & estimates                               &                                          &                                           &                       \\ \hline
  		$\EGNH(\alpha,\beta,a,b)$  & $2.41\times10^{-1}$                     & 1.194                                    & $1.27\times10^{-5}$                       & 14.268                \\
  		& $(1.12\times10^{-1})$                   & ($1.53\times10^{-1}$)                    & $(1.67\times10^{-6})$                     & ($3.94\times10^{-1}$) \\
  		BNH$(\alpha,\beta,a,b)$    & 1.189                                   & $1.16\times10^{-1}$                      & $2.48\times10^{-5}$                       & 6.886                 \\
  		& $(1.22\times10^{-1})$                   & $(6.14\times10^{-2})$                    & $(3.64\times10^{-6})$                     & $(3.94\times10^{-1})$ \\
  		KNH$(\alpha,\beta,a,b)$    & 2.202                                   & 9.559                                    & $6.629\times10^{-6}$                      & 5.903                 \\
  		& $(1.71\times10^{-1})$                   & $(3.44\times10^{-1})$                    & $(4.76\times10^{-7})$                     & $(3.89\times10^{-1})$ \\
  		MNH$(\alpha,a,b)$          & $6.08\times10^{-7}$                     & $2.55\times10^{-4}$                      & 0.157                                     &                       \\
  		& $(2.73\times10^{-8})$                   & $(5.67\times10^{-5})$                    & $(1.22\times10^{-4})$                     &                       \\
  		MONH$(\alpha,a,b)$         & 4.65                                    & $6.24\times10^{-6}$                      & 19.46                                     &                       \\
  		& $(6.85\times10^{-1})$                   & $(6.30\times10^{-7})$                    & $(1.59)$                                  &                       \\
  		ZBNH$(\alpha,a,b)$         & 1.81                                    & $2\times10^{-6}$                         & 57.17                                     &                       \\
  		& $(1.1\times10^{-1})$                    & $(2.93\times10^{-7})$                    & (1.81)                                    &                       \\
  		GPW$(a,b,c)$               & 3.57                                    & 1.37                                     & $7.95\times10^{-7}$                       &                       \\
  		& $(1.916\times10^{-14})$                 & $(6.13\times10^{-13})$                   & $(4.07\times10^{-8})$                     &                       \\
  		ENH$(\alpha,a,b)$          & 1.992                                   & $5.18\times10^{-6}$                      & 20.139                                    &                       \\
  		& $(1.45\times10^{-1})$                   & $(1.74\times10^{-7})$                    & $(4.57\times10^{-2})$                     &                       \\
  		NH$(a,b)$                  & $2.95\times10^{-6}$                     & 28.51                                    &                                           &                       \\
  		& $(1.51\times10^{-7})$                   & (1.01)                                   &                                           &                       \\
  		EE$(\alpha,\beta)$         & 3.02                                    & $2.16\times10^{-4}$                      &                                           &                       \\
  		& $(3.06\times10^{-1})$                   & $(1.02\times10^{-5})$                    &                                           &                       \\
  		Weibull$(a,b)$             & 2.16                                    & 9411.52                                  &                                           &                       \\
  		& $(1.19\times10^{-1})$                   & (341.45)                                 &                                           &                       \\ \hline
  	\end{tabular}
  	\label{table:EstimativesApplication2}
  \end{table}

  Based on the figures of merit in Table \ref{table:GoodnessOfFitApplication1}, the proposed EGNH distribution presents smallest GoF values. 
  Figure \ref{fig:Application1Densities} displays fitted and empirical densities for this dataset.
  We consider only EGNH, KNH and BNH models because these models present the smallest values of the KS statistic.
  The EGNH and KNH distributions seems to provide a closer fit to the histogram.
  The empirical cdf and the fitted cdfs of these models are displayed in Figure \ref{fig:Application1ecdf}.
  From this plot, we note that the EGNH model provides a good fit.

  On the other hand, with respect to Table \ref{table:GoodnessOfFitApplication2}, notice that the proposed EGNH model presents smallest values for the W*, A* and KS statistics.
  The fitted EGNH, MONH and Weibull pdfs in Figure \ref{fig:Application2Densities} reveal that the proposed model fits current data better than others.
  The use of AIC, BIC, CAIC and HQIC is only recommended to compare nested models, then it is not adequate to compare the Weibull and MONH models with the EGNH using these criteria.
  In Figure \ref{fig:Application2ecdf}, we can see that the empirical and fitted cdfs, which confirm that the EGNH model is one of the best candidates to describe the second dataset.
  In summary, the EGNH model may be an interesting alternative to other models available in the literature for modelling positive real data.

  \begin{table}[htb!]
    \centering
    \caption{Goodness-of-fit tests for \citeauthor{Aarset-HowtoIdentify-1987}'s dataset}
    \begin{tabular}{lrrrrrrr}
      \hline
      model                     & W*             & A*             & KS             & AIC             & CAIC            & BIC             & HQIC \\ \hline
      $\EGNH(\alpha,\beta,a,b)$ & \textbf{0.191} & \textbf{1.381} & 0.141          & \textbf{454.73} & \textbf{455.62} & \textbf{462.38} & \textbf{457.64} \\
      BNH$(\alpha,\beta,a,b)$   & 0.198          & 1.409          & 0.131          & 457.56          & 458.45          & 465.21          & 460.47 \\
      KNH$(\alpha,\beta,a,b)$   & 0.209          & 1.479          & \textbf{0.129} & 459.44          & 460.34          & 467.09          & 462.36 \\
      MNH$(\alpha,a,b)$         & 0.286          & 1.877          & 0.164          & 475.73          & 476.25          & 481.47          & 477.92 \\
      MONH$(\alpha,a,b)$        & 0.314          & 2.028          & 0.160          & 477.64          & 478.17          & 483.38          & 479.83 \\
      TNH$(\alpha,a,b)$         & 0.333          & 2.132          & 0.173          & 477.89          & 478.42          & 483.63          & 480.08 \\
      ENH$(\alpha,a,b)$         & 0.349          & 2.219          & 0.204          & 472.36          & 472.88          & 478.09          & 474.55 \\
      ZBNH$(\alpha,a,b)$        & 0.359          & 2.279          & 0.199          & 473.61          & 474.13          & 479.34          & 475.79 \\
      GPW$(a,b,c)$              & 0.372          & 2.341          & 0.194          & 475.89          & 476.42          & 481.63          & 478.07 \\
      NH$(a,b)$                 & 0.342          & 2.181          & 0.178          & 476.02          & 476.27          & 479.84          & 477.47 \\
      EE$(\alpha,\beta)$        & 0.485          & 2.950          & 0.204          & 483.99          & 484.25          & 487.81          & 485.45 \\
      Weibull$(a,b)$            & 0.496          & 3.008          & 0.193          & 486.00          & 486.26          & 489.83          & 487.46 \\ \hline
    \end{tabular}
    \label{table:GoodnessOfFitApplication1}
  \end{table}

  \begin{table}[htb!]
    \centering
    \caption{Goodness-of-fit tests for \citeauthor{AndrewsHerzberg-StressRuptureLife-1985}'s dataset}
    \begin{tabular}{lrrrrrrr}
      \hline
      model                     & W*             & A*             & KS             & AIC             & CAIC            & BIC             & HQIC \\ \hline
      EGNH$(\alpha,\beta,a,b)$ & \textbf{0.032} & \textbf{0.236} & \textbf{0.069} & 966.81          & 967.72          & 974.38          & 969.68 \\
      MONH$(\alpha,a,b)$        & 0.033          & 0.246          & 0.072          & \textbf{965.42} & \textbf{965.95} & 971.09          & 967.57 \\
      BNH$(\alpha,\beta,a,b)$   & 0.033          & 0.237          & 0.070          & 966.95          & 967.86          & 974.51          & 969.82 \\
      ZBNH$(\alpha,a,b)$        & 0.054          & 0.323          & 0.079          & 965.49          & 966.02          & 971.17          & 967.64 \\
      GPW$(a,b,c)$              & 0.057          & 0.361          & 0.191          & 968.59          & 969.12          & 974.26          & 970.74 \\
      ENH$(\alpha,a,b)$         & 0.058          & 0.383          & 0.084          & 966.26          & 966.80          & 971.94          & 968.42 \\
      NH$(a,b)$                 & 0.059          & 0.388          & 0.186          & 973.48          & 973.75          & 977.27          & 974.92 \\
      Weibull$(a,b)$            & 0.073          & 0.482          & 0.088          & 965.70          & 965.95          & \textbf{969.48} & \textbf{967.13} \\
      KNH$(\alpha,\beta,a,b)$   & 0.075          & 0.496          & 0.089          & 969.79          & 970.70          & 977.36          & 972.66 \\
      EE$(\alpha,\beta)$        & 0.164          & 1.061          & 0.118          & 972.14          & 972.40          & 975.92          & 973.58 \\
      MNH$(\alpha,a,b)$         & 0.342          & 2.137          & 0.277          & 1021.80         & 1022.34         & 1027.48         & 1023.96 \\
      \hline
    \end{tabular}
    \label{table:GoodnessOfFitApplication2}
  \end{table}

  \begin{figure*}[htb!] 
    \centering
    \subfigure[\citeauthor{Aarset-HowtoIdentify-1987}]{
      \includegraphics[width=0.45\textwidth]{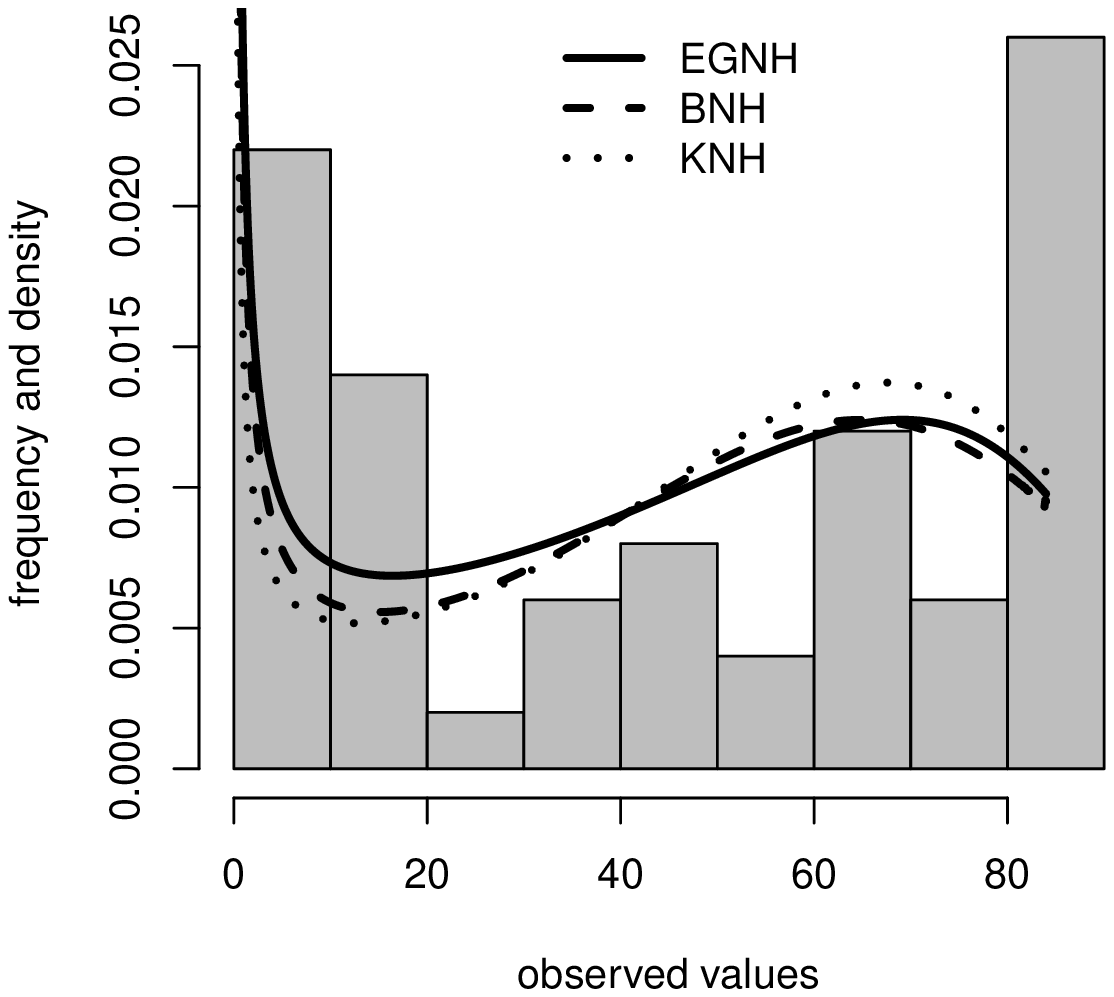}
      \label{fig:Application1Densities}
    }
    \subfigure[\citeauthor{AndrewsHerzberg-StressRuptureLife-1985}]{
      \includegraphics[width=0.45\textwidth]{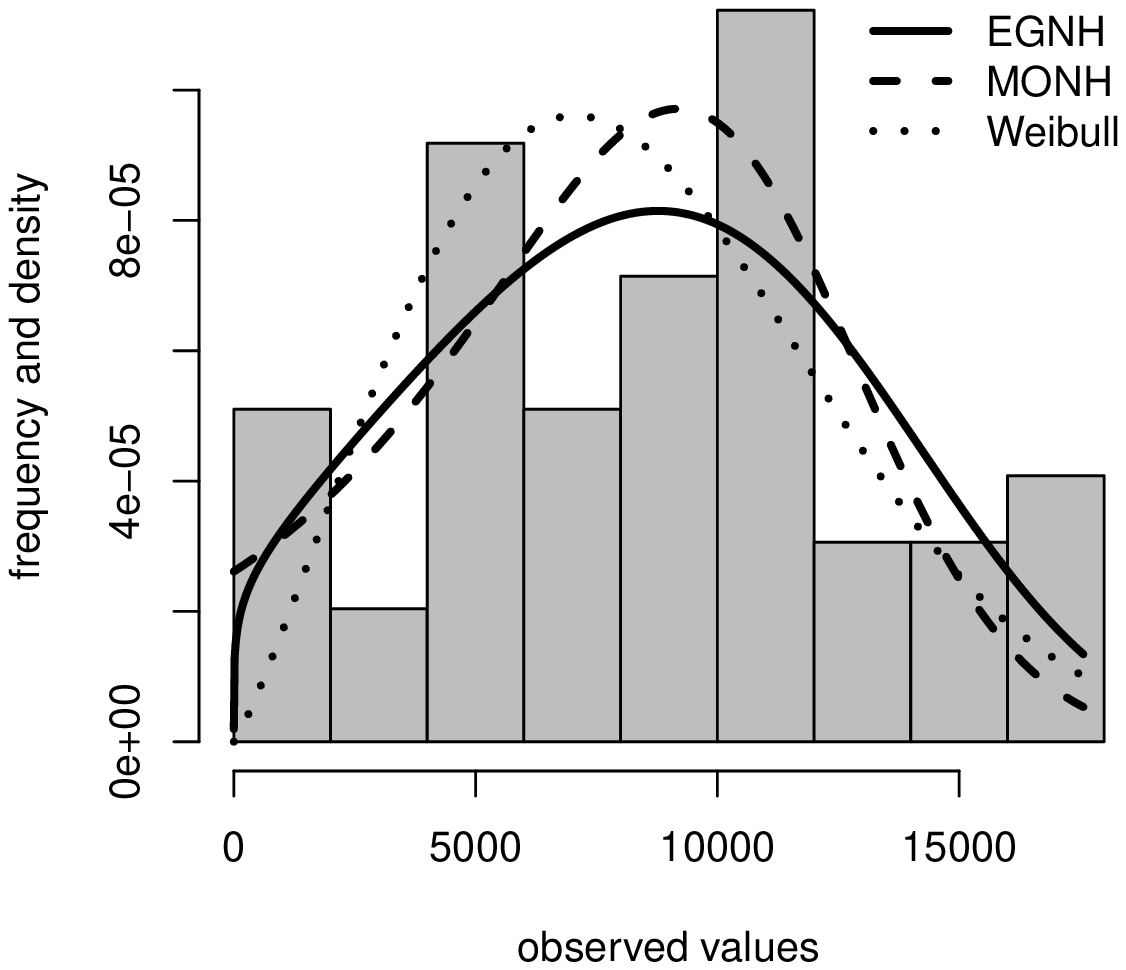}
      \label{fig:Application2Densities}
    }
    \caption{Empirical and fitted pdfs}
  \end{figure*}

  \begin{figure*}[htb!] 
    \centering
    \subfigure[\citeauthor{Aarset-HowtoIdentify-1987}]{
      \includegraphics[width=0.45\textwidth]{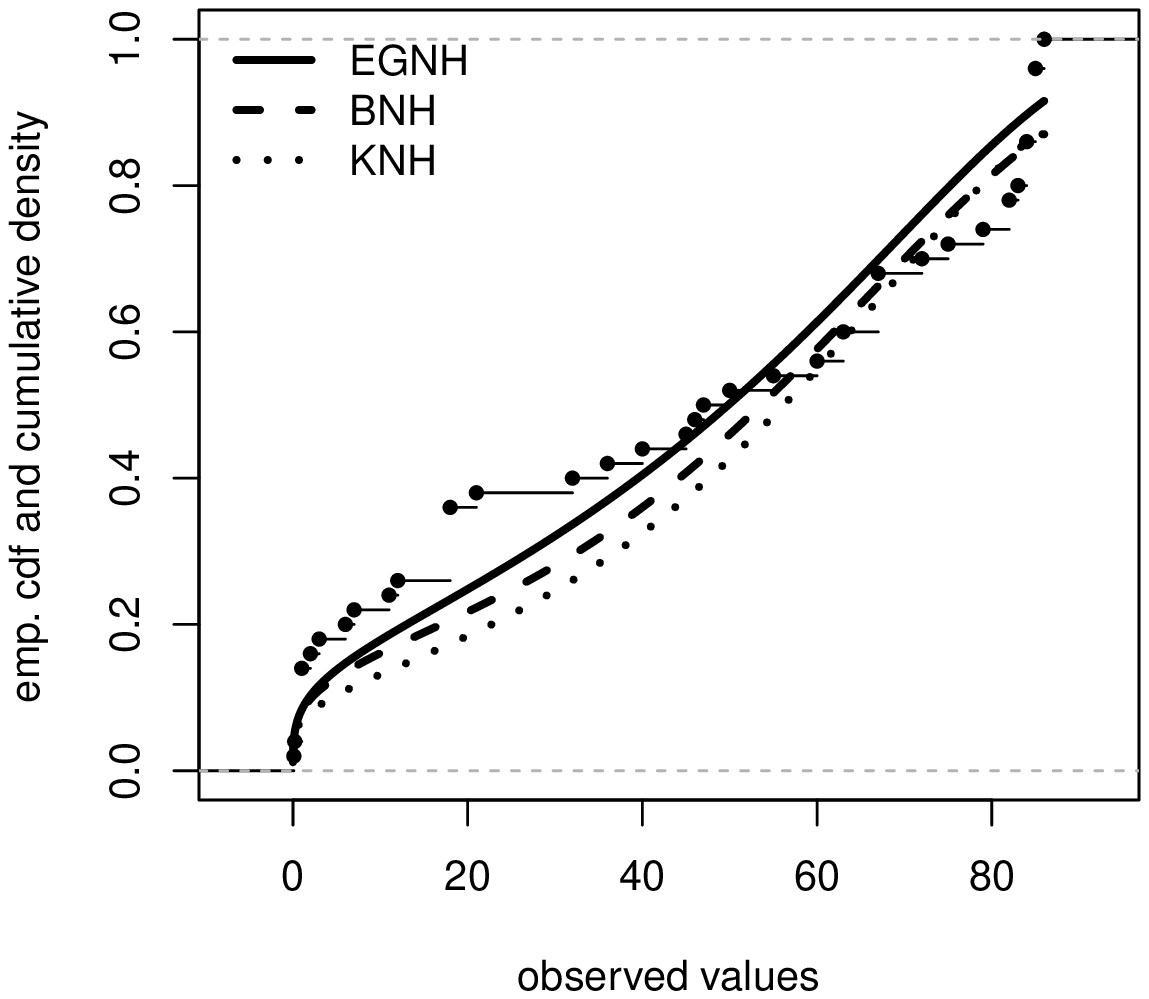}
      \label{fig:Application1ecdf}
    }
    \subfigure[\citeauthor{AndrewsHerzberg-StressRuptureLife-1985}]{
      \includegraphics[width=0.45\textwidth]{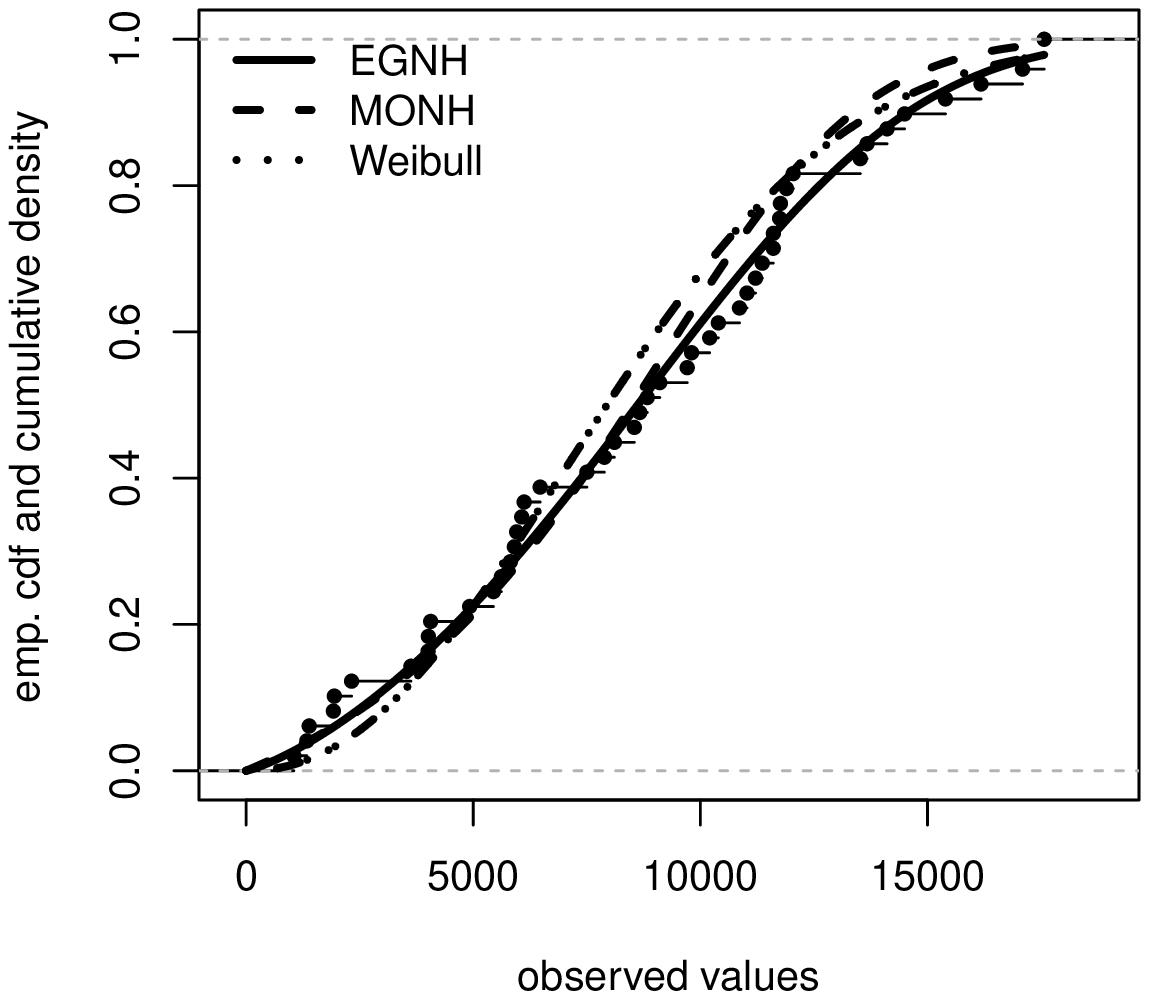}
      \label{fig:Application2ecdf}
    }
    \caption{Empirical and fitted pdfs}
  \end{figure*}

  \section{Concluding remarks}\label{sec:Concludingremarks}

  In this paper, we propose a new distribution, which is a simple generalization of the Nadarajah-Haghighi's ex\-po\-nen\-tial proposed by \citet{NadarajahHaghighi-extensionexponentialdistribution-2011}.
  Further, the new model includes others known distributions as special cases.
  We refer to it as the \emph{exponentiated generalized Nararajah-Haghighi} (EGNH) distribution.
  Some of its properties have been derived and discussed: ordinary and incomplete moments, mean deviations, Rényi entropy and order statistics.
  We present theoretical evidence that the new distribution can take decreasing, increasing, bathtub-shaped and upside-down bathtub hazard rate functions.
  We provide a maximum likelihood procedure for estimating the EGNH parameters.
  A Monte Carlo study to assess the consistency of the maximum likelihood estimates is performed, which satisfies the expected asymptotic behavior.
  The usefulness of the new model is illustrated through two applications to real data.
  The results indicate that our proposal can furnish better performance than other extended Nadarajah-Haghighi models recorded in the survival literature.
  The formulae related with the new model are manageable and may turn into adequate tools comprising the arsenal of applied statistics.
  We hope that the proposed model may attract wider applications for modeling positive real data in many areas such as engineering, survival analysis, hydrology and economics.

  \bibliographystyle{myapa}
  \bibliography{ArtigoEGNH-chjs.bib}

\end{document}